\documentclass[fontsize=11pt,a4paper,DIV=9,footnotes=multiple]{scrartcl}
\usepackage[utf8]{inputenc}
\usepackage[T1]{fontenc}
\usepackage[fleqn]{amsmath}
\usepackage{amsthm,amstext,amssymb,amsfonts,mathptmx}
\usepackage[babel,english=british]{csquotes}
\usepackage[english]{babel}
\usepackage[numbers]{natbib}
\usepackage{booktabs,calc,scrpage2,graphicx}
\usepackage{subcaption,caption}
\usepackage{enumitem}
\usepackage{hyperref}
\setkomafont{disposition}{\normalfont\bfseries}
\newtheoremstyle{normalstyle}
	{\baselineskip}
	{}
	{\normalfont}
	{}
	{\normalfont\bfseries}
	{.\newline}
	{.5\baselineskip}
	{}
\theoremstyle{normalstyle}
\newtheorem{thm}{Theorem}[section]
\newtheorem{lem}[thm]{Lemma}
\newtheorem{cor}[thm]{Corollary}
\newtheorem{definition}[thm]{Definition}

\newtheorem{ex}[thm]{Example}

\makeatletter
\renewcommand\p@subfigure{\thefigure\,}

\DeclareCaptionLabelFormat{mysublabelfmt}{(\alph{sub\@captype})}
\makeatother
\DeclareGraphicsExtensions{pdf,.eps,.png}

\newcommand{\vect}[1]{\ensuremath{\mathbf{#1}}}
\newcommand{\mat}[1]{\ensuremath{\mathbf{#1}}}
\newcommand{\dMatVec}[1]{\ensuremath{\mathcal{J}_{#1}}}
\newcommand{\gSet}[2][\empty]{
	\ensuremath{%
		\mathcal G_{\mathrm{#1}}(#2)
	}%
}
\newcommand{\pSet}[2][\empty]{
	\ensuremath{%
		\mathcal P_{\!\mathrm{#1}}(#2)
	}%
}
\newcommand{\pSetX}[2][\empty]{
	\ensuremath{%
		\mathcal P_{\!#1}(#2)
	}%
}
\newcommand{\lattice}[1]{\ensuremath{\Lambda(#1)}}
\newcommand{\translationOp}[1]{\ensuremath{\text{T}_{#1}}} 
\newcommand{\modulationOp}[1]{\ensuremath{\text{\^T}_{#1}}} \newcommand{\E}{\ensuremath{\mathrm{e}}}
\newcommand{\I}{\ensuremath{\mathrm{i}}}
\newcommand{\D}{\ensuremath{\,\mathrm{d}}}
\newcommand{\T}{\ensuremath{\mathrm{T}}}
\newcommand{\sumJMaskOp}[4][\empty]{\ensuremath{\Phi_{#4}\mathopen{#1(}#2,#3\mathclose{#1)}}}
\DeclareFontFamily{U}{mathx}{\hyphenchar\font45}
\DeclareFontShape{U}{mathx}{m}{n}{
      <5> <6> <7> <8> <9> <10>
      <10.95> <12> <14.4> <17.28> <20.74> <24.88>
      mathx10
      }{}
\DeclareSymbolFont{mathx}{U}{mathx}{m}{n}
\DeclareMathSymbol{\bigtimes}{1}{mathx}{"91}
\newcommand{\TitleText}{Multivariate periodic wavelets of de la Vallée Poussin type}
\title{\vspace{-1\baselineskip}\huge \TitleText}
\author{
Ronny Bergmann%
\thanks{Department of Mathematics, University of Technology Kaiserslautern, Paul-Ehrlich-Stra\ss e 31, D-67653 Kaiserslautern, Germany, bergmann@mathematik.uni-kl.de}
\and
J{\"u}rgen Prestin\thanks{Institute of Mathematics, University of Lübeck, Ratzeburger Allee 160, D-23562 Lübeck, Germany.\newline prestin@math.uni-luebeck.de}
}
\date{}
\hypersetup{pdfauthor={Ronny Bergmann and Jürgen Prestin},
	pdftitle={\TitleText},%
	pdfsubject={In this paper we present a general approach to multivariate periodic wavelets generated by scaling functions of de la Vallée Poussin type. These scaling functions and their corresponding wavelets are determined by their Fourier coefficients, which are sample values of a function, that can be chosen arbitrarily smooth, even with different smoothness in each direction. This construction generalizes the one-dimensional de la Vallée Poussin means to the multivariate case and enables the construction of wavelet systems, where the set of dilation matrices for the two-scale relation of two spaces of the multiresolution analysis may contain shear and rotation matrices. It further enables the functions contained in each of the function spaces from the corresponding series of scaling spaces to have a certain direction or set of directions as their focus, which is illustrated by detecting jumps of certain directional derivatives of higher order.},%
	plainpages=false,
	pdfstartview=FitH,
	pdfview=FitH,
	pdfpagemode=UseOutlines,
	bookmarksnumbered=true,
	bookmarksopen=false,
	bookmarksopenlevel=0,
	colorlinks=true,
	linkcolor=black,
	citecolor=black,
	urlcolor=black}

\begin{document}\thispagestyle{empty}
	\maketitle
	\begin{abstract}
	\noindent
	In this paper we present a general approach to multivariate periodic wavelets generated by scaling functions of de la Vallée Poussin type. These scaling functions and their corresponding wavelets are determined by their Fourier coefficients, which are sample values of a function, that can be chosen arbitrarily smooth, even with different smoothness in each direction. This construction generalizes the one-dimensional de la Vallée Poussin means to the multivariate case and enables the construction of wavelet systems, where the set of dilation matrices for the two-scale relation of two spaces of the multiresolution analysis may contain shear and rotation matrices. It further enables the functions contained in each of the function spaces from the corresponding series of scaling spaces to have a certain direction or set of directions as their focus, which is illustrated by detecting jumps of certain directional derivatives of higher order.
	\end{abstract}
	\paragraph{Keywords.} wavelets, lattices, de la Vallée Poussin means, periodic multiresolution analysis, shift-invariant space
	\paragraph{Mathematical Subject Classification 2010.} 42C40, 65T60
	\section{Introduction}
	In the recent years, a framework for multivariate wavelets was
	developed~\citep{Bergmann:2013,GohLeeTeo:1999,Koh:1995,LangemannPrestin:2010,MeximenkoSkopina03,Skopina:2011}, generalizing the one-dimensional periodic wavelets, which were developed in~\citep{NaWa1996,Skopina:2011,PlonkaTasche:1994}, to the multivariate case.
	By introducing the possibility to use any integer matrix for the underlying
	patterns, this framework possesses anisotropic approximation
	properties~\citep{BergmannPrestin:2013}. The directional approach was also
	discussed in the case of wavelets on the \( \mathbb R^d \) recently, e.g.
	for the shearlets~\cite{Dahlke:2009} in order to detect
	singularities~\cite{GuoLabate:2010}. For the one-dimensional case the de la
	Vallée Poussin means and their corresponding periodic wavelets create a
	multiresolution analysis (MRA) with a fast decomposition algorithm and
	provide good localization~\citep{PrestinSelig:1998,Selig:1998}. These de la
	Vallée Poussin wavelets were used in \citep{MhaskarPrestin:2000} in order to
	detect singularities of a function. In~\citep{LangemannPrestin:2010} an MRA
	was constructed using Dirichlet kernels and a certain set of dilation
	matrices. In~\citep{Bergmann:2013} fast algorithms for both the Fourier and
	wavelet transform were presented for any set of dilation matrices and
	corresponding scaling functions forming an MRA.

	Our main goal is the construction of multivariate trigonometric functions
	defined through their Fourier coefficients by sampling a function \( g \) of
	certain smoothness. This function \( g \) can be of compact support and the
	sampling points vary for each space of the nested sequence of spaces that
	form the MRA. We investigate, how to also obtain the coefficients of the
	two-scale relation from this construction. For a special case, \( g \)
	resembles not only the de la Vallée Poussin means of the one-dimensional
	case, but also the mentioned Dirichlet kernels. Choosing a smooth function \(g\) follows the idea of obtaining well localized scaling functions which also lead to well localized periodic wavelets. This construction further
	introduces the possibility to use arbitrary dilation matrices, especially
	shearing matrices such as \(%
	\mat{J}_{\text{X}}^{\pm}:=
	\bigl(\begin{smallmatrix} 2&0\\\pm1&1 \end{smallmatrix}\bigr)
	\) and \( \mat{J}_{\text{Y}}^{\pm}:=
	\bigl(\begin{smallmatrix} 1&\pm1\\0&2 \end{smallmatrix}\bigr) \).
	The presented approach does not only enable a construction for an arbitrary
	set of scaling functions, for the dyadic case we also obtain a similar
	construction for the corresponding wavelets. Their analogues to the
	two-scale relation, i.e. the coefficients that characterize the wavelet in
	the scaling space they are nested in, is derived. Both, the two-scale
	relation and these coefficients can be computed just using \( g \) and the
	dilation matrix \( \mat{J} \) involved.

	We investigate for which functions \( g \) we can derive a construction of a
	set of nested dyadic spaces where each scaling function \( \varphi_j \)
	together with its translates
	forms a basis for the corresponding space and
	depends on the next scaling function \( \varphi_{j+1} \) at most. On the one
	hand, the construction enables an explicit description of a complete MRA using the
	function \( g \). On the other hand, it increases the flexibility of an
	adaptive MRA, where the sequence of dilation matrices \( \mat{J}_j \) and
	hence their scaling functions \( \varphi_j \) and wavelets \( \psi_j \) can
	be chosen adaptively. The scaling spaces obtained in this setting can also be
	used to construct and characterize anisotropic interpolation operators.
	Corresponding error estimates and function spaces are discussed
	in~\citep{BergmannPrestin:2013}.
	
	The remainder of this paper is organized as follows. In
	Section~\ref{sec:MRA} we first introduce basic preliminaries in order to
	define the multiresolution analysis for an arbitrary sequence of dilation
	matrices \( \{\mat{J}_j\}_{j>0} \). For the dyadic case, i.e. where
	\(|\det\mat{J}_j|=2 \) for all \( j>0 \), Lemma~\ref{lem:orthNInTranslates}
	characterizes the orthonormality of the translates of the corresponding
	wavelet. In Lemmata \ref{lem:MSA:LU}-\ref{lem:MR3inck} the MRA is described,
	by using the Fourier coefficients \( c_{\vect{k}}(\varphi_j) \) of the 
	scaling functions \( \varphi_j \).
	
	In Section~\ref{sec:dlVP1D} we introduce the notation of the classical de la
	Vallée Poussin means from~\citep{PrestinSelig:1998,Selig:1998} and refine
	its notation in order to demonstrate the challenges arising by building a
	direct multivariate analogue.
	
	The functions of de la Vallée Poussin type based on a smooth function
	\( g \) are presented in Section~\ref{sec:dlVP}. For a sequence of \( n \)
	dilation matrices we obtain \( n+1 \) scaling functions. The spaces of their
	translates are characterized in Theorem~\ref{thm:dlVP}, especially the
	spaces are nested and the Fourier coefficients of these scaling functions
	are sample values of a function that possesses the same smoothness
	properties as \( g \). For the dyadic case, i.e. both the scaling space
	corresponding to \(j\) and its orthogonal complement with respect to the
	next space \(j+1\) are  again a spaces of translates of just one function
	each, we obtain a wavelet of de la Vallée Poussin type, which is
	investigated in Theorem~\ref{thm:dlVPW}. It inherits the properties from the
	corresponding scaling functions and is also given through Fourier
	coefficients by sampling a smooth function, too.
	
	Finally, in Section~\ref{sec:dlVP-MRA} we investigate for which functions
	\( g \) the dyadic scaling functions
	\( \varphi_{\mat{M}_l}^{\mathcal J_{l+1,n}} \) do not depend on the complete
	vector of matrices \( \mathcal J_{l+1,n} \), but only on the first matrix
	\( \mat{J}_{l+1} \). This generalizes the one-dimensional case, where only
	the scaling factor \( 2 \) was fixed. Starting from the two-dimensional
	case, we also investigate the \( d \)-dimensional case. Finally, in
	Section~\ref{sec:Example} the wavelets of de la Vallée Poussin type are
	illustrated by decomposing a two-dimensional box spline with two specific
	wavelets in order to detect singularities of its higher order directional
	derivatives.
\section{The multiresolution analysis}\label{sec:MRA}
	For a regular matrix \( \mat{M}\in\mathbb Z^{d\times d} \),
	\( d\in\mathbb N \), we define the lattice \( \lattice{\mat{M}}
	:= \mat{M}^{-1}\mathbb Z^d \) and obtain an equivalence relation with
	respect to \( \bmod \mathbf{1} \) on \( \lattice{\mat{M}} \), where
	\(\mathbf{1} \) denotes the vector \( (1,\ldots,1)^\T\in\mathbb R^d\).
	For a vector \(\vect{x}\in\mathbb R^d\) we denote the usual Euclidean norm
	by \(\|\vect{x}\|\). A pattern \( \pSet{\mat{M}} \) is defined as any set of
	representatives of\( \bmod \mathbf{1} \) on the lattice
	\( \lattice{\mat{M}} \). Any pattern \( \pSet{\mat{M}} \) equipped with the
	addition \( \bmod \mathbf{1} \) generates an abelian group. Especially with
	\( \mathcal Q_d := \bigl[-\tfrac{1}{2},\tfrac{1}{2}\bigr)^d \) the two sets 
	\[
		\pSet[S]{\mat{M}} := \mat{M}^{-1}\mathbb Z^d\cap \mathcal Q_d
		\quad\text{ and }\quad
		\pSet[I]{\mat{M}} := \mat{M}^{-1}\mathbb Z^d\cap [0,1)^d 
	\]
	are patterns. In the following we will omit the notation
	\( \bmod \mathbf{1}\), wherever it is clear from the context, that the
	addition of two elements \( \vect{x}+\vect{y} \) of a pattern is performed.

	For a multivariate \( 2\pi \)-periodic function
	\( f: \mathbb T^d\to\mathbb R \), where
	\( \mathbb T^d:=\mathbb R^d/2\pi\mathbb Z^d \) denotes the
	\( d \)-dimensional torus, we can apply the shift operator
	\( \translationOp{\vect{y}}f := f(\circ - 2\pi \vect{y}) \),
	\( \vect{y}\in\mathbb R^d \). 

	For any factorization \( \mat{M}=\mat{J}\mat{N} \), the pattern
	\( \pSetX[X]{\mat{N}} \) is a subset of \( \pSetX[X]{\mat{M}} \),
	\( X\in\{\text{S},\text{I}\} \). For this factorization we use the same
	notation for any set of representatives, i.e.
	\( \pSet{\mat{N}} \subset \pSet{\mat{M}} \).
	Using the congruence relation  on \( \mathbb Z^d \) with respect to a
	regular matrix \( \mat{M}\in\mathbb Z^{d\times d} \), which is defined by
	\[
		\vect{h}\equiv\vect{k}\bmod\mat{M}
		\Leftrightarrow
		\exists\,\vect{z}\in\mathbb Z^{d\times d}
		:
		\vect{k} = \vect{h}+\mat{M}^\T\vect{z},\quad \vect{k},\vect{h}\in\mathbb Z^d\text{,}
	\]
	we further define the generating set \( \gSet{\mat{M}} \) as any set of
	congruence representants of the equivalence relation \( \bmod\mat{M}  \).
	Due to the bijectivity of the linear map \( \mat{M}\circ \), a pattern
	always derives a generating set \( \gSet{\mat{M}}
	= \mat{M}\pSet{\mat{M}} \), in particular
	\[
		\gSet[S]{\mat{M}} := \mathbb Z^d\cap \mat{M}\mathcal Q_d
		\quad\text{ and }\quad
		\gSet[I]{\mat{M}} := \mathbb Z^d\cap \mat{M}[0,1)^d 
	\]
	are generating sets. When performing additions on the generating set
	\(\mathcal G(\mathbf{M})\), i.e. with respect to \(\bmod \mathbf{M}\), we
	omit the modulo, when it is clear from the context, that the result is again
	an element from \(\mathcal G(\mathbf{M})\). Using a geometrical argument
	\cite[Lemma~II.7]{deBoorHoelligRiemenscheider:1993}, it holds
	\[
		|\pSet{\mat{M}}|
		= |\pSet{\mat{M}^\T}|
		= |\gSet{\mat{M}}|
		= |\gSet{\mat{M}^\T}|
		= |\det\mat{M}| =:m\text{.}
	\]

	For any regular matrix \( \mat{M}\in\mathbb Z^{d\times d} \) the Fourier
	matrix \( \mathcal F(\mat{M}) \) is given by
	\begin{equation}\label{eq:def:fouriermatrix}
		\mathcal F (\mat{M})
		:=
		\frac{1}{\sqrt{m}}
		\left(
			\E^{-2\pi\I \vect{h}^\T\vect{y}}
		\right)_{\vect{h}\in\gSet{\mat{M}^\T},\vect{y}\in\pSet{\mat{M}}}
		\in \mathbb C^{m\times m}\text{,}
	\end{equation}
	where the pattern \( \pSet{\mat{M}} \) addressing the columns and the
	generating set \( \gSet{\mat{M}} \) addressing the rows are ordered in an
	arbitrary but fixed way. For any vector
	\( \vect{a} = \bigl(a_{\vect{y}}\bigr)_{\vect{y}\in\pSet{\mat{M}}} \) having
	the same order as the columns of the Fourier matrix
	\( \mathcal F(\mat{M}) \) the discrete Fourier transform with respect to
	\( \mat{M} \) is defined by
	\begin{equation*}
		\vect{\hat a}
		:= (\hat a_{\vect{h}})_{\vect{h}\in\gSet{\mat{M}^\T}}
		= \sqrt{m}\mathcal F(\mat{M})\vect{a}\in\mathbb C^{m}\text{,}
	\end{equation*}
	where \(\vect{\hat a} = (a_{\vect{h}})_{\vect{h}\in\gSet{\mat{M}^\T}}\) has
	the same order of the elements as the rows of \(\mathcal F(\mat{M})\). For a
	certain order of the elements of both sets, a fast Fourier transform
	exists~\citep[Section 4]{Bergmann:2013}.

	A subspace \( V \) of functions of the Hilbert space \( L_2(\mathbb T^d) \)
	on the torus \( \mathbb T^d \) is called shift-invariant with respect to
	\( \mat{M} \), or \( \mat{M} \)-invariant, if
	\[
		\forall \varphi \in V,\ \vect{y}\in\pSet{\mat{M}}
		\,:\,\translationOp{\vect{y}}\varphi 
		:= \varphi(\circ-2\pi\vect{y}) \in V
		\text{.}
	\]
	An important example of an \( \mat{M} \)-invariant space is given by the
	span of translates of \( \varphi \in L_2(\mathbb T^d) \), i.e.
	\[
		V_{\mat{M}}^{\,\varphi}
		:= \operatorname{span}
		\bigl\{ 			\translationOp{\vect{y}}\varphi\,:\,\vect{y}\in\pSet{\mat{M}}
		\bigr\}\text{,}
	\]
	where each function \( f\in V_{\mat{M}}^{\,\varphi} \) can be written as
	\begin{equation}\label{eq:coeff:a}
		f =
			\sum_{\vect{y}\in\pSet{\mat{M}}}a_{\vect{y}}
			\translationOp{\vect{y}}\varphi
			,
			\quad
			a_{\vect{y}}\in\mathbb C
			\text{.}
	\end{equation}
	This can also be equivalently formulated using the scalar product of the
	Hilbert space \( L_2(\mathbb T^d) \), the Fourier coefficients
	\[
		c_{\vect{k}}(\varphi) := \langle \varphi, \E^{\I \vect{k}^\T\circ}\rangle
		=
		\frac{1}{(2\pi)^d}\int_{\mathbb T^d}
		\varphi(\vect{x})\E^{\I\vect{k}^\T\vect{x}}\D\vect{x}
		,\quad
		\vect{k}\in\mathbb Z^d\text{,}
	\]
	and the Parseval equation, as~\citep[Theorem 3.3]{LangemannPrestin:2010}
	\begin{equation}\label{eq:coeff:hata}
		c_{\vect{h}+\mat{M}^\T\vect{z}}(f) = \hat a_{\vect{h}}c_{\vect{h}+\mat{M}^\T\vect{z}}(\varphi)
		\quad
		\text{ holds for all }
		\vect{h}\in\gSet{\mat{M}},\vect{z}\in\mathbb Z^d
		\text{,}
	\end{equation}
	where \( \vect{\hat a}
	= (\hat a_{\vect{h}})_{\vect{h}\in\gSet{\mat{M}^\T}} \) is the discrete
	Fourier transform of a vector of coefficients \( \vect{a}
	= (a_{\vect{y}})_{\vect{y}\in\pSet{\mat{M}}} \) from~\eqref{eq:coeff:a}. 
		
	For any factorization \( \mat{M} = \mat{J}\mat{N} \) the sub-pattern
	\( \pSet{\mat{N}} \) also induces a subspace, i.e. for a function
	\( \eta\in V_{\mat{M}}^{\varphi} \) we obtain
	\( V_{\mat{N}}^{\,\eta}\subset V_{\mat{M}}^{\,\varphi} \). Using
	the coefficients characterizing \(\eta\) with respect to the translates
	of \(\varphi\), cf. \citep[Lemma 4.2 (i) and (ii)]{LangemannPrestin:2010},
	we can characterize each function \(f\in V_{\mathbf{N}}^{\,\eta}\) also in
	\(V_{\mat{M}}^{\varphi} \). We extend this characterization by the following
	Lemma to investigate the orthogonality of the translates
	\( \translationOp{\vect{x}}\eta \), \( \vect{x}\in\pSet{\mat{N}} \), of the
	subspace.
	\begin{lem}\label{lem:orthInTranslates}
		Let \( \mat{M} = \mat{J}\mat{N} \) and a function
		\( \varphi\in L_2(\mathbb T^d) \) be given, such that the shifts
		\( \translationOp{\vect{y}}\varphi \), \( \vect{y}\in\pSet{\mat{M}} \),
		are an orthonormal basis of the \( \mat{M} \)-invariant space
		\( V_{\mat{M}}^{\,\varphi} \). Let \( \eta
		= \sum_{\vect{y}\in\pSet{\mat{M}}} b_{\vect{y}}\translationOp{\vect{y}}\varphi \in V_{\mat{M}}^{\,\varphi} \)
		be given. Then the translates \( \translationOp{\vect{x}}\eta \),
		\( \vect{x}\in\pSet{\mat{N}} \), are orthonormal if and only if
		\[
			\sum_{\vect{g}\in\gSet{\mat{J}^\T}}
			|\hat b_{\vect{k}+\mat{N}^\T\vect{g}}|^2 = |\det\mat{J}|
			\quad
			\text{ holds for all }
			\vect{k}\in\gSet{\mat{N}^\T}
			\text{,}
		\]
		where \( \vect{\hat b} \) is the discrete Fourier Transform of
		\( \vect{b} = \bigl(b_{\vect{y}}\bigr)_{\vect{y}\in\pSet{\mat{M}}} \).
	\end{lem}
	\begin{proof}
		Let \(m=\lvert \det \mat{M}\rvert\) and \(n=\lvert \det \mat{N}\rvert\). Using \citep[Corollary 3.6]{LangemannPrestin:2010},~\eqref{eq:coeff:hata}
		and the orthonormality of the translates
		\( \translationOp{\vect{y}}\varphi \), \( \vect{y}\in\pSet{\mat{M}} \),
		the orthonormality of the translates \( \translationOp{\vect{x}}\eta \),
		\( \vect{x}\in\pSet{\mat{N}} \), \(\eta\in V_{\mat{M}}^{\,\varphi}\) as
		above, is equivalent to the fact that it holds for all
		\( \vect{k}\in\gSet{\mat{N}^\T} \)
		\begin{align*}
			|\det{\mat{J}}| = 
			\frac{m}{n} &=  m\sum_{\vect{w}\in\mathbb Z^d}
			|c_{\vect{k}+\mat{N}^\T\vect{w}}(\eta)|^2\\
			&= m\sum_{\vect{g}\in\gSet{\mat{J}^\T}}
				|\hat b_{\vect{k}+\mat{N}^\T\vect{g}}|^2
				\sum_{\vect{z}\in\mathbb Z^d}
				|c_{\vect{k}+\mat{N}^\T\vect{g}+\mat{M}^\T\vect{z}}(\eta)|^2\\
			&= \sum_{\vect{g}\in\gSet{\mat{J}^\T}}
			|\hat b_{\vect{k}+\mat{N}^\T\vect{g}}|^2\text{.}
			\qedhere
		\end{align*}
	\end{proof}
	If \( f\in V_{\mat{M}}^{\,\varphi} \) is known by its coefficient vector
	\( \vect{a}\in\mathbb C^m \) from~\eqref{eq:coeff:a}
	, then the projection
	\( f_{\mat{N},\eta} \) of \( f \) onto \( V_{\mat{N}}^{\,\eta} \) is given
	in matrix form in \citep[Lemma 4.2 (iii)]{LangemannPrestin:2010} and can be
	reformulated as~\citep[(1.50)]{Bergmann:2013b}
	\begin{equation}\label{eq:projection}
		\hat{d}_{\vect{k}}
		=
		\frac{1}{\sqrt{|\det{\mat{J}}|}}
		\sum_{\vect{g}\in\gSet{\mat{J}^\T}}
		\overline{\hat a_{\vect{k+\mat{N}^\T\vect{g}}}}
		\hat b_{\vect{k}+\mat{N}^\T\vect{g}}
		\text{,}
		\quad
		\vect{k}\in\gSet{\mat{N}^\T}
		\text{,}			
	\end{equation}
	such that \( f_{\mat{N},\eta}
	= \sum_{\vect{x}\in\pSet{\vect{N}}}d_{\vect{x}}
		\translationOp{\vect{x}}\eta \), where again
	\( \vect{\hat d}\in\mathbb C^n \) 
	is the discrete Fourier transform (with respect to \( \mat{N} \)) of
	\( \vect{d}\in\mathbb C^n \). 

	Let the translates \( \translationOp{\vect{y}}\varphi \),
	\(\vect{y}\in\pSet{\mat{M}}\), be linearly independent. Let further a set of
	functions \( \eta_1,\ldots,\eta_{|\det{\mat{J}}|}\in V_{\mat{M}}^{\varphi} \)
	be given such that the translates \( \translationOp{\vect{x}}\eta_j \) are
	linearly independent for each \( j=1,\ldots,|\det\mat{J}| \). If these
	translates are mutually orthogonal, i.e.
	\(\langle\translationOp{\vect{x}}\eta_i,\translationOp{\vect{y}}\eta_j\rangle
	=0 \) for each \( i \neq j \), \( i,j\in \{1,\ldots,|\det\mat{J}|\}\), and
	\( \vect{x}\neq\vect{y} \), \( \vect{x},\vect{y}\in\pSet{\mat{N}} \), we
	obtain an orthogonal decomposition of the space
	\( V_{\mat{M}}^{\,\varphi} \) into
	\begin{equation}
		\label{eq:decomp:VM}
		V_{\mat{M}}^{\,\varphi} = \bigoplus_{j=1}^{|\det{\mat{J}}|} V_{\mat{N}}^{\,\eta_j}\text{.}
	\end{equation}
	By applying~\eqref{eq:projection} to each of the subspaces
	\( V_{\mat{N}}^{\,\eta_j} \), we obtain a decomposition algorithm in
	\( \mathcal O(m) \) steps, cf.~\cite[Subsection 5.3]{Bergmann:2013}.

	For the case of a dyadic decomposition, i.e. \(|\det{\mat{J}}|=2 \), the
	following Lemma states how to construct the orthogonal complement in the
	direct sum of~\eqref{eq:decomp:VM}. While Theorem 4.3
	of~\citep{LangemannPrestin:2010} presents a general construction for the
	translates of the orthogonal complement, this lemma further characterizes,
	how to obtain these translates \(T_{\vect{x}}\psi\),
	\(\vect{x}\in\pSet{\mat{N}}\), as an orthonormal basis of the corresponding
	space.
	\begin{lem}\label{lem:orthNInTranslates}
		Let \( \mat{M}=\mat{J}\mat{N} \) be a decomposition of the regular matrix
		\( \mat{M}\in\mathbb Z^{d\times d} \) into integer matrices
		\( \mat{N},\,\mat{J}\in\mathbb Z^{d\times d} \), where
		\(|\det\mat{J}|=2 \). Let \( \varphi,\eta\in L_2(\mathbb T^d) \) be two
		functions, such that the translates
		\( \translationOp{\vect{y}}\varphi \),\( \vect{y}\in\pSet{\mat{M}} \),
		and \( \translationOp{\vect{x}}\eta \), \( \vect{x}\in\pSet{\mat{N}} \),
		are orthonormal bases of \( V_{\mat{M}}^{\,\varphi} \) and
		\( V_{\mat{N}}^{\,\eta} \) respectively, where
		\[
			c_{\vect{h}+\mat{M}^\T\vect{z}}(\eta)
			= 
			\hat{a}_{\vect{h}}
			c_{\vect{h}+\mat{M}^\T\vect{z}}(\varphi)
			,\quad
			\vect{h}\in\gSet{\mat{M}^\T},\ 
			\vect{z}\in\mathbb Z^d
			\text{.}
		\]
		Then the orthogonal complement of \( V_{\mat{N}}^{\,\eta} \) in
		\( V_{\mat{M}}^{\,\varphi} \) is again a shift-invariant space of the form
		\( V_{\mat{N}}^{\,\psi} \) possessing the orthonormal basis
		\( \translationOp{\vect{x}}\psi \), \( \vect{x}\in\pSet{\mat{N}} \), if
		and only if there exist numbers
		\(\sigma_{\vect{h}}\in\mathbb C\backslash\{0\} \), which fulfill
		\begin{align*}
			\sigma_{\vect{k}} &= -\sigma_{\vect{k}+\mat{N}^\T\vect{g}}
			\text{ for }\vect{g}\in\gSet{\mat{J}^\T}\backslash\{\vect{0}\},
			\vect{k}\in\gSet{\mat{N}^\T}
			\intertext{ and }
			|\sigma_{\vect{h}}| &= 1\text{ for }\vect{h}\in\gSet{\mat{M}^\T}\text{.}
			\end{align*}
			The function \( \psi \) is given by its Fourier coefficients as
			\[
			c_{\vect{h}+\mat{M}^\T\vect{z}}(\psi)
			=
			\sigma_{\vect{h}}\hat{a}_{\vect{h}+\mat{N}^\T\vect{g}}
			c_{\vect{h}+\mat{M}^\T\vect{z}}(\varphi)
			,\quad
			\vect{h}\in\gSet{\mat{M}^\T}, \ 
			\vect{z}\in\mathbb Z^d
			\text{.}
			\]
		\end{lem}
		\begin{proof}
		Using the orthonormality of the translates
		\( \translationOp{\vect{y}}\varphi \), \( \vect{y}\in\pSet{\mat{M}} \),
		and \citep[Theorem~4.3]{LangemannPrestin:2010} we obtain that the
		translates \( \translationOp{\vect{x}}\psi \),
		\( \vect{x}\in\pSet{\mat{N}} \), span the orthogonal complement of
		\( V_{\mat{N}}^{\,\eta} \) in \( V_{\mat{M}}^{\,\varphi} \) if and only
		if for  \( \psi
		= \sum_{\vect{y}\in\pSet{\mat{M}}}
			b_{\vect{y}}\translationOp{\vect{y}}\varphi \) there exist
		coefficients \( \sigma'_{\vect{h}} \in\mathbb C\backslash\{0\} \) with
		\( \sigma'_{\vect{k}} = -\sigma'_{\vect{k}+\mat{N}^\T\vect{g}} \),
		\( \vect{g}\in\gSet{\mat{J}^\T}\backslash\{\vect{0}\} \),
		\( \vect{k}\in\gSet{\mat{N}^\T} \), such that
		\[
			\hat b_{\vect{h}}
			=
			\frac{\sigma'_{\vect{h}}
				\overline{\hat{a}_{\vect{h}+\mat{N}^\T\vect{g}}}%
			}{
				\sum\limits_{\vect{k}\in\mathbb Z^d}
					|c_{\vect{h}+\mat{M}^\T\vect{k}}(\eta)|^2
			}
		=
		m\sigma'_{\vect{h}}\overline{\hat{a}_{\vect{h}+\mat{N}^\T\vect{g}}},
		\quad\text{for all } \vect{h}\in\gSet{\mat{M}^\T}\text{.}
		\]
		Applying Lemma~\ref{lem:orthInTranslates} we conclude that the translates
		\( \translationOp{\vect{x}}\psi \), \( \vect{x}\in\pSet{\mat{N}} \), are
		orthonormal if and only if it holds for all
		\( \vect{k}\in\gSet{\mat{N}^\T} \)
		\begin{equation*}
			2 = |\hat b_{\vect{k}}|^2 + |\hat b_{\vect{k}+\mat{N}^\T\vect{g}}|^2
			  = m^2|\sigma'_{\vect{k}}|^2
				\bigl(
					|\hat a_{\vect{k}+\mat{N}^\T\vect{g}}|^2+|\hat a_{\vect{k}}|^2
				\bigr)
			  = 2m^2|\sigma'_{\vect{k}}|^2\text{,}
		\end{equation*}
		hence setting \( \sigma_{\vect{h}} = m\sigma'_{\vect{h}} \),
		\( \vect{h}\in\gSet{\mat{M}^\T} \), finishes the proof.
	\end{proof}
	\begin{definition}
		For a sequence \( \{\mat{J}_k\}_{k>0} \) of regular matrices
		\( \mat{J}_k\in\mathbb Z^{d\times d} \), \( |\det\mat{J}_k|>1 \),
		and a sequence of spaces \( \{V_j\}_{j\geq0} \),
		\( V_j\subset L_2(\mathbb T^d) \), we denote
		\( \mat{M}_0 := \mat{E}_d \), where
		\(\mat{E}_d\in\mathbb Z^{d\times d}\) is the unit matrix,
		\( m_0=1 \), and for \( j>0 \)
		\[
			\mat{M}_j
			:= \mat{J}_j\mat{M}_{j-1}
			= \prod_{k=j}^1\mat{J}_k = \mat{J}_{j}\cdot\ldots\cdot\mat{J}_1
			\quad
			\text{and}
			\quad
			m_j := |\det{\mat{M}_j}|\text{.}
		\]
		An anisotropic periodic multiresolution analysis of
		\( L_2(\mathbb T^d) \) is given by the tuple
		\( \bigl(\{\mat{J}_{k}\}_{k>0},\allowbreak\,\{V_{j}\}_{j\geq 0}\bigr) \)
		if the following properties hold
		\begin{enumerate}[label={MR}\arabic*,leftmargin=2.5em]
			\item For \( j\in \mathbb N_0 \), there exists a function
			\( \varphi_j\in V_j \), such that the translates
			\( \translationOp{\vect{y}}\varphi_j\),
			\( \vect{y}\in\pSet{\mat{M}_j} \), constitute a basis of \( V_j \).
			\label{MSA:Basis}
			\item For all \(j\in \mathbb N\), we have \(V_j \subset V_{j+1}\).
			\label{MSA:Subsets}
			\item The union of all \( V_j \) is dense in \( L_2(\mathbb T^d) \).
			\label{MSA:complete}
		\end{enumerate}
	\end{definition}
	The following Lemmata characterize these three properties of an anisotropic
	periodic multiresolution analysis
	\( \bigl(\{\mat{J}_{k}\}_{k>0},\allowbreak\,\{V_{j}\}_{j\geq 0}\bigr) \)
	using the Fourier coefficients \( c_{\vect{k}}(\varphi_j) \) of the scaling
	functions \( \varphi_j \). This was already examined for the one-dimensional
	case in \citep{Skopina:2011,Selig:1998}.
	\begin{lem}\label{lem:MSA:LU}
		For \( j\in\mathbb N \) the property \ref{MSA:Basis} is equivalent to
		\( \dim V_j = m_j \) and  the existence of a function
		\( \varphi_j\in V_j \), such that
		\begin{equation}\label{eq:MSA1:j}
			\sum_{\vect{z}\in\mathbb Z^d}
				|c_{\vect{h}+\mat{M}_j^\T\vect{z}}(\varphi_j)|^2 > 0
				\quad
				\text{holds for all }\vect{h}\in\gSet{\mat{M}_j^\T}\text{.}
		\end{equation}
	\end{lem}
	\begin{proof}
		Let~\ref{MSA:Basis} be given. Then for each \(j\in\mathbb N_0\) the
		equality \(\dim V_j = m_j\) holds and~\eqref{eq:MSA1:j} is a
		reformulation of \eqref{eq:coeff:hata}. For the reverse direction of the
		equivalence, the property given in~\eqref{eq:MSA1:j} together with the
		just mentioned dimension of \(V_j\) ensures, that the translates
		\( \translationOp{\vect{y}}\varphi_j\in V_j \),
		\( \vect{y}\in\pSet{\mat{M}_j} \), are a basis of \(V_j\).
	\end{proof}
	\begin{lem}\label{lem:MSA:Subsets:inck}
		Let \ref{MSA:Basis} be fulfilled for a set of regular matrices
		\( \{\mat{J}_k\}_{k>0} \) and a set of spaces \( \{V_j\}_{j\geq 0} \).
		Then, the property \ref{MSA:Subsets} holds if and only if for all
		\( j\in\mathbb N \) there exists a vector
		\( \vect{\hat a}_j = \bigl(
			\hat a_{j,\vect{h}}
		\bigr)_{\vect{h}\in\gSet{\mat{M}^\T_{j+1}}} \in \mathbb C^{m_{j+1}} \),
		such that
		\begin{equation}\label{eq:InSpaceVj}
			c_{\vect{h}+\mat{M}_{j+1}^\T\vect{z}}(\varphi_j)
			=
			\hat a_{j,\vect{h}}
			c_{\vect{h}+\mat{M}_{j+1}^\T\vect{z}}(\varphi_{j+1})
			\quad\text{holds for all }\vect{h}\in\gSet{\mat{M}_{j+1}^\T},
			\ \vect{z}\in\mathbb Z^d
			\text{.}
		\end{equation}
	\end{lem}
	\begin{proof}
		Due to \( \mat{M}_{j+1} = \mat{J}_{j+1}\mat{M}_j \) we have
		\( \pSet{\mat{M}_j} \subset \pSet{\mat{M}_{j+1}} \).
		Using \ref{MSA:Basis} the property \ref{MSA:Subsets} is equivalent to
		fulfilling a statement like \eqref{eq:coeff:hata} for each
		\( j\in \mathbb N \), which is stated in \eqref{eq:InSpaceVj}.
	\end{proof}
	\begin{cor}
		For two successive scaling functions, Lemma~\ref{lem:MSA:Subsets:inck}
		implies, that for the sequences
		\(\vect{c}(\varphi_j)
		:= \{c_{\vect{k}}(\varphi_j)\}_{\vect{k}\in\mathbb Z^d}\)
		and \(\vect{c}(\varphi_{j+1})\) it holds
		\[
			\operatorname{supp}(\vect{c}(\varphi_j))
			\subseteq
			\operatorname{supp}(\vect{c}(\varphi_{j+1}))
			,
			\quad
			j\in\mathbb N_0
			\text{.}
		\]
	\end{cor}
	\begin{lem}\label{lem:MR3inck}
		Let \ref{MSA:Basis} and \ref{MSA:Subsets} be fulfilled.
		The property~\ref{MSA:complete} holds if and only if
		\begin{equation}\label{eq:SuppPhijZd}
			\lim_{j\to\infty} \operatorname{supp}(\vect{c}(\varphi_j))
			= \mathbb Z^d\text{.}
		\end{equation}
	\end{lem}
	\begin{proof}
		The proof follows the one-dimensional approach, mentioned
		in~\citep{PlonkaTasche:1993,PlonkaTasche:1994}.

		Assume,~\eqref{eq:SuppPhijZd} is not fulfilled. Then there exists a
		vector \( \vect{k}_0\in\mathbb Z^d \), such that
		\( c_{\vect{k}_0}(\varphi_j)=0 \) is true for all \( j\in\mathbb N \).
		Using the basis \( \{\E^{\I\vect{k}^\T\circ}
		\,;\,\vect{k}\in\mathbb Z^d\}\) of the Hilbert space
		\( L_2(\mathbb T^d) \), we obtain, that the function
		\( \E^{\I\vect{k}_0^\T\circ} \) is orthogonal to all spaces \( V_j \).
		Hence, the union of these spaces is not dense in \( L_2(\mathbb T^d) \)
		and~\ref{MSA:complete} does not hold.
		
		Assume,~\eqref{eq:SuppPhijZd} is fulfilled. In order to deduce,
		that~\ref{MSA:complete} holds, assume there exists a function
		\( f\in L_2(\mathbb T^d) \), \( f\neq 0 \), such that
		\begin{equation}\label{eq:proof:f-orth}
			f\ \bot\ \operatorname{clos}_{L_2(\mathbb T^d)}
			\Biggl(\bigcup_{j\in\mathbb N_0}V_j\Biggr)
		\end{equation}
		holds. There exists \( \vect{k}_0\in\mathbb Z^d \) with
		\[
			|c_{\vect{k}_0}(f)|
			= \max \{ |c_{\vect{k}}(f)|\ ;\ \vect{k}\in\mathbb Z^d\} > 0\text{.}
		\]
		Using \eqref{eq:SuppPhijZd} there exists a \( j_0\in\mathbb N_0 \), such
		that \( \vect{k}_0 \in \operatorname{supp}(\vect{c}(\varphi_{j_0})) \).
		Furthermore using~\ref{MSA:Subsets}, we have
		\( \varphi_{j_0} \in V_{j}\), for each \(j\geq j_0 \) and hence
		\( V_{\mat{M}_{j_0}}^{\varphi_{j_0}}\subset V_j \).
		Due to \eqref{eq:proof:f-orth} we have \( f \bot V_j \) and especially
		\( f\bot V_{\mat{M}_{j_0}}^{\varphi_{j_0}} \). Using the discrete
		Fourier transform of the  coefficients \( a_{\vect{y}} \) characterizing
		the projection of \( f \) into \( V_{\mat{M}_{j}}^{\varphi_{j}} \), we
		obtain for each \( j\geq j_0\), \( \vect{h}\in\gSet{\mat{M}_j^\T} \),
		that
		\begin{equation*}
			0 = \sum_{\vect{y}\in\pSet{\mat{M}_j}}
				\langle f,\translationOp{\vect{y}}\varphi_{j_0}\rangle
				\E^{-2\pi\I\vect{h}^\T\vect{y}}
			  = m_j\sum_{\vect{z}\in\mathbb Z^d}
				c_{\vect{h}+\mat{M}_j^\T\vect{z}}(f)
				\overline{c_{\vect{h}+\mat{M}_j^\T\vect{z}}(\varphi_{j_0})}
				\text{.}
		\end{equation*}	
		Looking at the congruence classes, we see, that we have a unique vector
		\( \vect{h}\in\gSet{\mat{M}_j^\T} \) fulfilling
		\( \vect{h}\equiv\vect{k}_0\bmod\mat{M}_j^\T\) for any \(j\geq j_0 \). 		The corresponding sum can be written as
		\begin{equation}\label{eq:proof:M3inck:splitsum}
			0 = c_{\vect{k}_0}(f)\overline{c_{\vect{k}_0}(\varphi_{j_0})}
			+ \sum_{\vect{z}\in\mathbb Z^d\backslash\{\vect{0}\}}
				c_{\vect{k}_0+\mat{M}_j^\T\vect{z}}(f)
				\overline{c_{\vect{k}_0+\mat{M}_j^\T\vect{z}}(\varphi_{j_0})}
				,\quad j\geq j_0\text{.}
	\end{equation}
		Let \( \varepsilon := |c_{\vect{k}_0}(f)
		\overline{c_{\vect{k}_0}(\varphi_{j_0})}| > 0 \). Applying the Cauchy
		Schwarz inequality to the series of the absolute values of the Fourier
		coefficients, yields
		\begin{equation*}
			\sum_{\vect{z}\in\mathbb Z^d}
				|c_{\vect{z}}(f)\overline{c_{\vect{z}}(\varphi_{j_0})}|
				\leq \|\vect{c}(f)\ |\ \ell_2\|
				\cdot\|\vect{c}(\varphi_{j_0})\ |\ \ell_2\| < \infty\text{.}
		\end{equation*}
		Hence, there exists a \( j_1 \geq j_0 \), such that
		\begin{equation*}
			\sum_{\vect{z}\in\mathbb Z^d\backslash\gSet{\mat{M}_{j_1}^\T}}
				|c_{\vect{z}}(f)\overline{c_{\vect{z}}(\varphi_{j_0})}|
			< \varepsilon/2\text{,}
	\end{equation*}
		which yields a contradiction, choosing \( j=j_1 \)
		in~\eqref{eq:proof:M3inck:splitsum} and hence \( f=0 \). This implies,
		that~\ref{MSA:complete} follows from~\eqref{eq:SuppPhijZd} and completes
		the proof.
	\end{proof}
\section{De la Vallée Poussin means}\label{sec:dlVP1D}
	In \cite{PrestinSelig:1998,Selig:1998} the de la Vallée Poussin means
	\(\varphi_N^T\) were used to generate an one-dimensional MRA. They are given
	by their Fourier coefficients
	\begin{equation}\label{eq:1DdlVP}
		\varphi_N^T = \frac{1}{\sqrt{N}}\sum_{k\in\mathbb Z}
			c_k\E^{\I k\circ},
		\quad\text{where}\quad c_k =
		\begin{cases}
			1 &\mbox{ if } |k| < \frac{N}{2}-T\text{,}\\
			\frac{1}{2T}\bigl(\frac{N}{2}+T-|k|\bigr)
				&\mbox{ if } \frac{N}{2}-T \leq |k| \leq \frac{N}{2}+T\text{,}\\
			 0 &\mbox{ if } \frac{N}{2}+T < |k|\text{,}
		\end{cases}
	\end{equation}
	where \( N\in2\mathbb N \) is even and \(\tfrac{N}{2}\geq T\in\mathbb N\).
	For \( M=2N \) the two de la Vallée Poussin means \( \varphi_N^T \) and
	\( \varphi_M^S \), \( S \leq N \), provide two nested spaces of shifts,
	i.e.\ \( V_N^{\varphi_N^T} \subset V_M^{\varphi_M^S} \), if and only if
	\(S+T\leq \frac{N}{2} \), cf.~\cite[Theorem~4.1.3]{Selig:1998}.
\begin{figure}[tbp]
	\centering
	\includegraphics{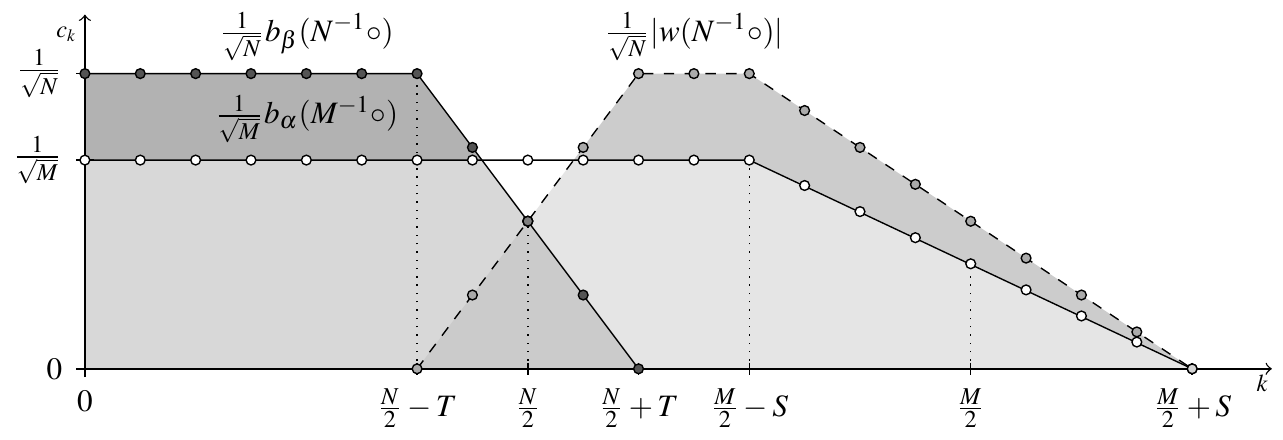}
	\caption[]{The Fourier coefficients with nonnegative \(k\in\mathbb N\) of the functions \(\varphi_N^T\) (dark), \(\varphi_M^S\) (white) and their corresponding wavelet \(\psi\) (gray, in absolute value) are shown for \( M=32=2N\), \(S=4=2T \). The corresponding scaled functions \(b_\alpha, b_\beta\) and \(\lvert w\rvert\) are also plotted.}\label{fig:dLVP1D}
\end{figure}
	They can also be obtained by using a scaled version of a linear
	spline, i.e.\ 
	\begin{align*}
		b_{\alpha}(x) &= \begin{cases}
			1 &\mbox{ if } |x| < \frac{1}{2}-\alpha\text{,}\\
			\frac{1}{2\alpha}\bigl(\frac{1}{2}+\alpha-|x|\bigr)
			&\mbox{ if } \frac{1}{2}-\alpha \leq |x|
				\leq \frac{1}{2}+\alpha\text{,}\\
			0 &\mbox{ if } \frac{1}{2}+\alpha  < |x|\text{,}
		\end{cases}
		\quad
		\text{ where } 0 <\alpha\leq \frac{1}{2}\text{,}
		\intertext{by defining}
		c_k(\varphi_{N}^\T)
		&=\tfrac{1}{\sqrt{N}}b_{\tfrac{T}{N}}(\tfrac{k}{N}),
		\quad k\in\mathbb Z\text{.}
	\end{align*}
	While the value \( T \) provides an absolute number of coefficients that
	decay, \( \alpha \) describes this as a relative part of \( N \).
	Let \( b_0 \) denote the limit \(b_0 = \lim_{\alpha\to0}b_{\alpha}\).
	Then, both extremal cases provide known functions, i.e. the modified
	Dirichlet kernel for \(b_0\) and the Féjer kernel for \(b_{\frac{1}{2}}\).
	In this notation the requirement for the spaces to be nested reads
	\( \frac{2S}{M}+\frac{T}{N} \leq \frac{1}{2} \), which can also be
	reformulated as \(2\alpha+\beta \leq \frac{1}{2}\) by setting
	\(\alpha := \frac{S}{M}\) and \(\beta := \frac{T}{N}\).
	Then we can provide the function
	\begin{equation*}
		w(x) = w_{\alpha,\beta}(x)
		:= \E^{-\frac{\I x\pi}{4}}
			\Bigr(b_{\alpha}\bigl(\tfrac{x}{2}\bigr)-b_{\beta}(x)\!\Bigr)
		= \E^{-\frac{\I x\pi}{4}}
			\biggl(\sum_{z\in\mathbb Z}
			b_{\beta}(x+1 + 2z)\!
			\biggr)
			b_{\alpha}\bigl(\tfrac{x}{2}\bigr)
	\end{equation*}
	that gives rise to a definition of the corresponding wavelet
	\( \psi = \psi_{M,N}^{S,T} \). The equality of both formulas for \(w(x)\) is a direct consequence from \citep[Theorem 4.2.1]{Selig:1998}.
	For \( \alpha=\beta \) this is similar to the function used in the Remark on
	page 29 in \citep{MhaskarPrestin:2000}. The wavelet \( \psi \) is given by
	sampling the function \(w\) on the points \( \frac{k}{N} \), i.e.
	\[
		c_{k}(\psi) = \tfrac{1}{\sqrt{N}}w_{\frac{S}{M},\frac{T}{N}}(N^{-1}k),
		\quad k\in\mathbb Z\text{.}
	\]
	Figure~\ref{fig:dLVP1D} illustrates the whole construction, where due to symmetry, we omitted the negative axis \(k<0\). We took \(M=32=2N\) and \(S=4=2T\), hence \(\alpha = \beta = \frac{1}{8}\). In order to illustrate \(w\) and the complex valued Fourier coefficients of the wavelet \(\psi\),, it's absolute value is used in the Figure.
	
	The inequality \( \frac{2S}{M}+\frac{T}{N} \leq \frac{1}{2} \) can also be seen in terms of the frequencies which the wavelet has to
	reproduce exactly. The inequality ensures, that at least
	\(
		\exp\bigl(\I\bigl(\frac{M+N}{4}+\frac{T-S}{2}\bigr)\circ\bigr)
		\in V_N^\psi
	\)\!.

	The advantage of this notation with a relative decay
	\( \frac{T}{N}=\alpha \) is that it is independent of the chosen
	\( N \). In the inequality \(2\alpha+\beta \leq \frac{1}{2}\)
	the factor of dilation can be seen in the first relative factor. This can
	easily be generalized to the multivariate functions on the torus
	\( \mathbb T^d \) using
	\begin{equation}\label{eq:TPLinearDecay}
		B_{\vect{\alpha}}(\vect{x})
		:=
		\prod_{j=1}^d b_{\alpha_j}(x_j)
		, \quad\vect{\alpha} = (\alpha_1,\ldots,\alpha_d)^\T\in
		\bigl[0,\tfrac{1}{2}\bigr]^d
		,\ \vect{x}\in\mathbb R^d
		\text{,}
	\end{equation}
	to define for any regular matrix \( \mat{M}\in\mathbb Z^{d\times d} \) a
	de la Vallée Poussin mean scaling function
	\( \varphi_{\mat{M}}^{\vect{\alpha}} \) via its Fourier coefficients
	\[
		c_{\vect{k}}(\varphi_{\mat{M}}^{\vect{\alpha}})
		:= \frac{1}{\sqrt{|\det{\mat{M}}|}}
		B_{\vect{\alpha}}(\mat{M}^{-\T}\vect{k}),
		\quad \vect{k}\in\mathbb Z^d\text{.}
	\]
	Any of the congruence class computations in the argument of
	\( B_{\alpha} \), i.e. decompositions like
	\( \vect{k}=\vect{h}+\mat{M}^\T\vect{z} \), where
	\( \vect{h}\in\gSet{\mat{M}^T} \), \( \vect{z}\in\mathbb Z^d \), can be
	performed by using the pattern \( \pSet{\mat{M}^T} \) and the congruence
	\(\bmod \mathbf{1}\) in each dimension.

	For \( \vect{\alpha}=\vect{0} \) we obtain the Dirichlet kernels
	\( D_{\mat{M}}^{\text{re}} \) from~\citep[Eq.~(34)]{LangemannPrestin:2010}.
	They form a dyadic MRA using the scaling matrices
	\[
		\mat{J}_k\in
		\{\mat{J}_{\text{D}},\mat{J}_{\text{X}},\mat{J}_{\text{Y}}\}
		:= \Biggl\{
		\begin{pmatrix} 1&-1\\1&1 \end{pmatrix},
		\begin{pmatrix} 2&0\\0&1 \end{pmatrix},
		\begin{pmatrix} 1&0\\0&2 \end{pmatrix}
		\Biggr\}
		\text{.}
		\]
	For \( \vect{\alpha}\neq \vect{0} \), this construction does not lead to
	an MRA in general.
	\begin{figure}[tbp]
		\begin{subfigure}[b]{.49\textwidth}\centering
			\includegraphics{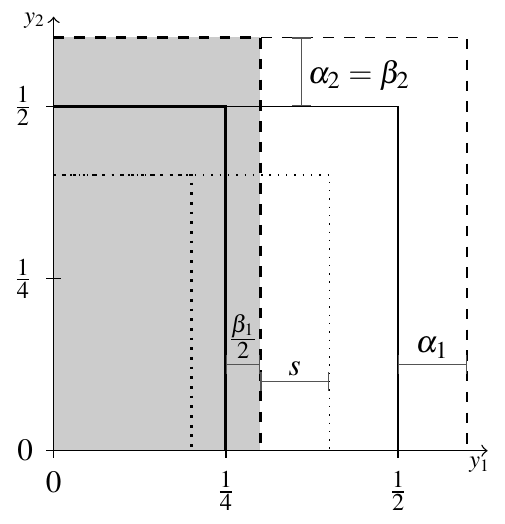}
			\caption{\( \mat{J}_X \)}\label{subfig:Balphabeta:Jx}
		\end{subfigure}
		\begin{subfigure}[b]{.49\textwidth}\centering
			\includegraphics{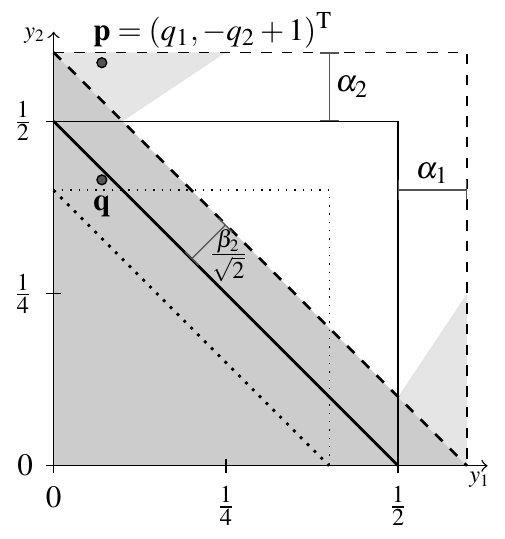}
		\caption{\( \mat{J}_D \)}\label{subfig:Balphabeta:Jd}
	\end{subfigure}
	\caption[]{Two cases of functions \( B_{\vect{\alpha}} \) (white) and
	\( B_{\vect{\beta}}(\mat{J}^\T\circ) \) (dark) for the
	matrices~\subref{subfig:Balphabeta:Jx} \( \mat{J}_X \)
	and~\subref{subfig:Balphabeta:Jd} \( \mat{J}_D \), each having
	\(\vect{\alpha}=\vect{\beta}=\frac{1}{10}(1,1)^\T\). The light area is the
	part, where frequencies of both kernels create the contradiction,
	e.g.\ looking at the points \( \vect{p} \) and \( \vect{q} \).}
	\label{fig:Balphabeta}
	\end{figure}
	As an example let us fix \( d=2 \) and
	\( \vect{\alpha}=\vect{\beta}=\tfrac{1}{10}(1,1)^\T \).
	Figure~\ref{fig:Balphabeta} illustrates the two cases
	\( B_{\vect{\beta}}(\mat{J}_{\text{X}}\vect{x}) \) and
	\( B_{\vect{\beta}}(\mat{J}_{\text{D}}\vect{x}) \) in comparison to
	\( B_{\vect{\alpha}}(\vect{x}) \). For simplification and due to symmetry,
	both figures are restricted to the first quadrant, where the dashed line
	delimits the support of \( B_{\vect{\alpha}} \) or the shaded support of
	\( B_{\vect{\beta}} \) and the dotted line marks the inner plateau. The
	first case using \( \mat{J}_{\text{X}} \) is similar to the one-dimensional
	case, because
	\( \beta_2=\alpha_2 \) and \( 2\alpha_1+\beta_1 = \tfrac{3}{10}
	\leq \tfrac{1}{2} \).

	In the second case, cf. Figure~\ref{subfig:Balphabeta:Jd}, we have for any
	point \( \vect{q} \), \( 0< q_1 < 2\alpha_1 \),
	\( \tfrac{1}{2} < q_2 < \tfrac{1}{2}+\alpha_2-|\alpha_1-q_1|\) that
	\( B_{\vect{\alpha}}(\vect{q})\neq 0 \) and
	\( B_{\vect{\beta}}(\mat{J}_{\text{D}}\vect{q}) \neq 0 \). In order to
	fulfill~\ref{MSA:Subsets}, cf. Lemma~\ref{lem:MSA:Subsets:inck}, these
	coefficients \( \hat{a}_{\vect{h}} \) are in this illustration --- due to
	multiplication with \( \mat{M}^{-\T} \) --- given by sampling a
	\( 1 \)-periodic function. For the point \( \vect{q} + (0,1)^\T  \) we have
	by symmetry, \( B_{\vect{\alpha}}(\vect{q}+(0,1)^\T)
	=B_{\vect{\alpha}}(\vect{p})\neq 0 \) and
	\( B_{\vect{\beta}}(\mat{J}_{\text{D}}(\vect{q}+(0,1)^\T)
	= B_{\vect{\beta}}(\mat{J}_{\text{D}}\vect{p})=0 \). This holds for any of
	the mentioned points \( \vect{q} \) and yields a contradiction
	to~\ref{MSA:Subsets}, if only the absolute value of the determinant of
	\( \mat{M} \) is large enough, such that any point
	\( \mat{M}^{-\T}\vect{k}\), \(\vect{k}\in\gSet{\mat{M}^\T} \),
	corresponds to one point \( \vect{q} \).
\section{Scaling functions of de la Vallée Poussin type}\label{sec:dlVP}
	This section is devoted to a construction of scaling functions and in the
	dyadic case their wavelets having an arbitrary smooth decay in their Fourier
	coefficients.
	\begin{definition}\label{def:admissible}
		We call a function \( g: \mathbb R^d\to\mathbb R \) \emph{admissible} if
		the following two properties are fulfilled
		\begin{enumerate}[label=F\arabic*)]
			\item\label{prop:g:pos} \( g(\vect{x}) \geq 0 \),
			for all \(\vect{x}\in\operatorname{supp} g\subset\mathbb R^d\),
			and \( g(\vect{x})>0\), \(\vect{x}\in\mathcal Q_d \),
			\item\label{prop:g:sum}\( \sum\limits_{\vect{z}\in\mathbb Z^d}
				g(\vect{x}+\vect{z}) = 1\), for all \(\vect{x}\in\mathbb R^d \).
		\end{enumerate}
	\end{definition}
	A special class of admissible functions are the centered box splines
	\( M_{\Xi}^c \), cf.~\citep[Chapter I]{deBoorHoelligRiemenscheider:1993},
	where the \( d\)-dimensional unit matrix \( \mat{E}_d \) is a subset of the
	vectors in \( \Xi \). The function \( B_{\vect{\alpha}} \)
	from~\eqref{eq:TPLinearDecay} is a special case of the centered box splines,
	more precisely with \( \Xi =
	\bigl(
		\vect{e}_1,\ldots,\vect{e}_d,\alpha_1\vect{e}_1,
		\ldots,\alpha_d\vect{e}_d
	\bigr) \), where \(\vect{e}_j\) denotes the \(j\)th unit vector.
	Of course, an admissible function \( g \) can also be chosen, such that it
	is arbitrarily smooth, e.g. by adding the last \( d \) vectors multiple
	times to the matrix \( \Xi \). Introducing the possibility to choose a smooth function \(g\) follows the idea, that sampling a smooth function to obtain the Fourier coefficients yields a certain localization, which was used in the one-dimensional case for example in~\cite{MhaskarPrestin:2000}.

	For a regular matrix \( \mat{J}\in\mathbb Z^{d\times d} \), a function
	\( f:\mathbb R^d \to \mathbb C \) is called summable with respect to
	\( \mat{J}\), if for all \( \vect{x}\in\mathbb R^d \)
	\begin{equation*}
		f^{\mat{J}}(\vect{x})
		:= \sum_{\vect{z}\in\mathbb Z^d}f(\vect{x}+\mat{J}^\T\vect{z})
		< \infty\text{.}
	\end{equation*}
	Any admissible function \(g\) is summable with respect to \(\mat{J}\).
	
	For any two functions \( f_1,f_2: \mathbb R^d\to\mathbb C \) and a regular
	matrix \( \mat{J}\in\mathbb Z^{d\times d} \), where \( f_1 \) is summable
	with respect to \( \mat{J} \), we define the operator
	\begin{equation*}
		\sumJMaskOp{f_1}{f_2}{\mat{J}}
		:= f_1^{\mat{J}}(\circ)
		f_2(\mat{J}^{-\T}\circ)
		= \Biggr(\sum_{\vect{z}\in\mathbb Z^d}
			f_1(\circ-\mat{J}^\T\vect{z})\Biggr)f_2(\mat{J}^{-\T}\circ) 
	\end{equation*}
	and for a vector of regular integral matrices
	\(\mat{J}_1,\ldots,\mat{J}_n\in\mathbb Z^{d\times d}\),
	\(n\in\mathbb N \), we define the matrix vectors
	\[
		 \dMatVec{l,k} := 
		 \begin{cases}
			 \bigl(
			 \mat{J}_j
			 \bigr)_{j=l}^k \in \bigl(\mathbb Z^{d\times d}\bigr)^{k-l+1}
			 &\mbox{ if } 1\leq l\leq k \leq n\text{,}\\
			 \emptyset&\mbox{else,}
		\end{cases}
	\]
	where we restrict the set of matrices by writing
	\( \dMatVec{l,k}\in \mathcal X^{k-l+1} \) for matrices \( \mat{J}_j\),
	\(j=l,\ldots,k \) from the set \( \mathcal X \). A matrix vector
	\(\dMatVec{l,k} \) only consisting of matrices having
	\(|\det\mat{J}_j|=2 \), \( j=1,\ldots,n \), is called dyadic.

	For a matrix vector \( \dMatVec{1,n} \) the functions
	\( B_{\!\!\dMatVec{l,k}}\), \(k\in\{1,\ldots,n\} \) are recursively defined
	by
	\[
		B_{\!\!\dMatVec{l,k}} := \begin{cases}
			g&\mbox{ if } l \geq k+1\text{,}\\
			\sumJMaskOp[\big]{g}{B_{\!\!\dMatVec{l+1,k}}}{\mat{J}_l}
			&\mbox{ if } l = k, k-1,\ldots,1\text{,}
		\end{cases}
	\]
	where \( g: \mathbb R^d\to \mathbb C\) is an admissible function.
	\begin{definition}\label{def:dlVP}
		Let \( \mat{M}_0 \in \mathbb Z^{d\times d} \) be a regular matrix,
		\( m_0 = |\det\mat{M}_0| > 0 \), and let
		\( \dMatVec{1,n}\in \bigl(\mathbb Z^{d\times d}\bigr)^n \),
		\( n\in \mathbb N \), be a vector of regular matrices \( \mat{J}_l \),
		\( l=1,\ldots,n \). We further denote for each
		\( l=1,\ldots,n \), that
		\( \mat{M}_l := \mat{J}_l\cdot\ldots\cdot\mat{J}_1\mat{M}_0 \) and
		\( m_l := |\det\mat{M}_l| \).

		The functions \( \varphi_{\mat{M}_l}^{\dMatVec{l+1,n}} \),
		\( l=0,\ldots,n\), which are defined by their Fourier coefficients
		\begin{align*}
			c_{\vect{k}}\bigl(\varphi_{\mat{M}_l}^{\dMatVec{l+1,n}}\bigr)
			&:= 			\frac{1}{\sqrt{m_l}}B_{\!\!\dMatVec{l+1,n}}(\mat{M}_l^{-\T}\vect{k}),
			\quad \vect{k}\in\mathbb Z^{d\times d},\ l=0,\ldots,n\text{,}
		\end{align*}
		are called scaling functions of de la Vallée Poussin type.
	\end{definition}
	We further introduce the corresponding spaces, which we denote by
	\[
		V_{\mat{M}_l}^{\dMatVec{l+1,n}}
		:= \operatorname{span}\Bigl\{
		\translationOp{\vect{y}}
		\varphi_{\mat{M}_{l}}^{\dMatVec{l+1,n}}
		\ ;\ 
		\vect{y}\in\pSet{\mat{M}_{l}}
		\Bigr\}
		,\quad
		l\in\{0,\ldots,n\}\text{.}
	\]
	Let \( C^r(\mathbb T^d) \), \( r\in\mathbb N \), denote the space of
	functions \( f:\mathbb T^d\to\mathbb C \) whose \( r \)th directional
	derivatives are continuous.
	\begin{thm}\label{thm:dlVP}
	Let the scaling functions of de la Vallée Poussin type
	\( \varphi_{\mat{M}_l}^{\dMatVec{l+1,n}}\), \(l=0,\ldots,n\),
	\( n\in\mathbb N \), be given as in Definition~\ref{def:dlVP}.
	Then the following statements hold
	\begin{enumerate}[label=\alph*)]
		\item\label{thm:dlVP:subSpace} The spaces
		\( V_{\mat{M}_l}^{\dMatVec{l+1,n}} \) are nested, i.e.\ 
		\[
			\varphi_{\mat{M}_l}^{\dMatVec{l+1,n}} \in 
			V_{\mat{M}_{l+1}}^{\dMatVec{l+2,n}}
			\quad\text{ holds for } l=0,\ldots,n-1\text{.}
		\]
		\item\label{thm:dlVP:lu} For each \( l\in\{0,\ldots,n\} \), the shifts
		\(
			\translationOp{\vect{y}}\varphi_{\mat{M}_{l}}^{\dMatVec{l+1,n}}
		\), 
		\(
			\vect{y}\in\pSet{\mat{M}_l}
		\),
		are linearly independent.
		\item\label{thm:dlVP:Sampling} Let \( g\in C^{r}(\mathbb R^d), r\in\mathbb N \). Then, for \( l\in\{0,\ldots,n\} \) it holds
		\( B_{\!\!\dMatVec{l+1,n}}(\mat{M}_l^{-\T}\circ)
			\in C^r(\mathbb R^d)\text{.} \) 
		\end{enumerate}
	\end{thm}
	\begin{proof}
		\begin{enumerate}[label=\alph*)]
			\item Let \( l\in\{0,\ldots,n-1\} \) be given. We define
			\begin{equation}\label{eq:twoscale}
				 \hat a_{l,\vect{h}} := \sqrt{|\det\mat{J}_{l+1}|}g^{\mat{J}_{l+1}}(\mat{M}_{l}^{-\T}\vect{h}),\quad \vect{h}\in\gSet{\mat{M}^\T_{l+1}}
			\end{equation}
			and use the unique decomposition of any \( \vect{k}\in\mathbb Z^d \) into \( \vect{k}=\vect{h}+\mat{M}_{l+1}^\T\vect{z}\), \( \vect{h}\in\gSet{\mat{M}_{l+1}^\T}\), \( \vect{z}\in\mathbb Z^d \), to obtain
		\begin{equation*}
			\begin{split}
				c_{\vect{h}+\mat{M}_{l+1}^\T\vect{z}}\bigl(\varphi_{\mat{M}_l}^{\dMatVec{l+1,n}}\bigr)
				&= \frac{1}{\sqrt{m_j}}B_{\!\!\dMatVec{l+1,n}}\Bigl(\mat{M}_{l}^{-\T}\bigl(\vect{h}+\mat{M}_{l+1}^\T\vect{z}\bigl)\Bigl)\\
				&= \frac{1}{\sqrt{m_j}}\sumJMaskOp{g}{B_{\!\!\dMatVec{l+2,n}}}{\mat{J}_{l+1}}\bigl(\mat{M}_l^{-\T}\vect{h}+\mat{J}_{l+1}^\T\vect{z}\bigr)\\
				&= \hat a_{l,\vect{h}}
					\frac{1}{\sqrt{m_{l+1}}}
					B_{\!\!\dMatVec{l+2,n}}\bigl(\mat{M}_{l+1}^{-\T}\vect{h}+\vect{z}\bigr)\\
				&= \hat a_{l,\vect{h}} c_{\vect{h}+\mat{M}_{l+1}^\T\vect{z}}\bigl(\varphi_{\mat{M}_{l+1}}^{\dMatVec{l+2,n}}\bigr)
				\text{.}
			\end{split}
		\end{equation*}
		Hence the statement follows with~\eqref{eq:coeff:hata}.
		%
		%
		\item For \( \vect{h}\in\gSet[S]{\mat{M}_l^\T} \) and \( l\in\{0,\ldots,n\} \) it holds \( c_{\vect{h}}(\varphi_{\mat{M}_l}^{\dMatVec{l+1,n}}) = \frac{1}{\sqrt{m_l}}B_{\!\!\dMatVec{l+1,n}}(\mat{M}_l^{-\T}\vect{h}) \) and \( \mat{M}_{l}^{-\T}\vect{h}\in\mathcal Q_d\). The linear independence of translates does not depend on the choice of the generating set \( \gSet{\mat{M}_l^\T} \). We will restrict the rest of the proof to the generating set \( \gSet[S]{\mat{M}_l^\T} \).

		For \( l=n \) we have for \( \vect{x}\in
		\mathcal Q_d \), that \( B_{\!\!\dMatVec{n+1,n}}(\vect{x}) = g(\vect{x}) > 0 \) and hence  \( c_{\vect{h}}(\varphi_{\mat{M}_n}^{\dMatVec{n+1,n}}) \neq 0 \) for \( \vect{h}\in\gSet[S]{\mat{M}_n^\T} \). 

		For \( 0\leq l<n \) the statement follows by induction over \( l \). By induction hypothesis \( B_{\!\!\dMatVec{l+2,n}}(\vect{x}) > 0 \) for \( \vect{x}\in
		\mathcal Q_d \) and for each \( \vect{y}\in\mathbb R^d \) there exist \( \vect{x}'\in\mathcal Q^d, \vect{z}'\in\mathbb Z^d \) such that \( \vect{y} = \vect{x}'+\vect{z} \). Hence if \( \vect{y} = \mat{J}_{l+1}^{-\T}\vect{x}\notin\mathcal Q_d\), we have \( \vect{x}'\in\mathcal Q_d \). The shift by \( \vect{z} \) does not affect the corresponding first factor because we have \( \vect{y}'=\vect{x} - \mat{J}_{j+1}^\T\vect{z} \). We obtain
		\begin{align*}
			B_{\!\!\dMatVec{l+1,n}}(\vect{x}) 
			&= \sumJMaskOp{g}{B_{\!\!\dMatVec{l+2,n}}}{\mat{J}_{l+1}}(\vect{x})\\
			&= \Biggl(\sum_{\vect{z}\in\mathbb Z} g(\vect{x} + \mat{J}^\T_{l+1}\vect{z})\Biggr)B_{\!\!\dMatVec{l+2,n}}(\mat{J}_{l+1}^{-\T}\vect{y}') >0\text{,}
		\end{align*}
		because \( g \) is nonnegative. Hence \(c_{\vect{h}}\bigl(\varphi_{\mat{M}_l}^{\dMatVec{l+1,n}}\bigr) \neq 0 \) for all \( \vect{h}\in\gSet[S]{\mat{M}_l^\T} \) and the translates \(\translationOp{\vect{y}}\varphi_{\mat{M}_{l}}^{\dMatVec{l+1,n}} \), \( \vect{y}\in\pSet[S]{\mat{M}_l} \) are linearly independent by \citep[Corollary 3.5]{LangemannPrestin:2010}.
		\item For \( l=n \) the function \( B_{\!\!\dMatVec{l+1,n}}(\mat{M}_l^{-\T}\circ)  \) is just a scaled, sheared and rotated version of the function \( g \) and hence by assumption in \( C^r(\mathbb R^d) \). For \( 0\leq l<n \) we obtain the statement by applying the same induction as in \ref{thm:dlVP:lu}.\qedhere
		\end{enumerate}
	\end{proof}
	The coefficients \( \hat{a}_{l,\vect{h}} \), \( l<n \), defined in the proof Theorem~\ref{thm:dlVP}\,\ref{thm:dlVP:subSpace} can also be interpreted as sampling the function \( \sqrt{|\det\mat{J}_{l+1}|}g^{\mat{J}_{l+1}}(\mat{M}_{l}^{-\T}\circ) \) at a subset of the integer vectors. As long as only these coefficients are needed, e.g. for the orthonormalization utilizing Lemma~\ref{lem:orthNInTranslates}, they can easily be obtained using the summation with respect to \( \mat{J}^\T_{l+1} \) of the function \( g \), which is analogously to Theorem~\ref{thm:dlVP}\,\ref{thm:dlVP:Sampling} a function in \( C^r(\mathbb R) \).
\begin{lem}\label{lem:SpezialDirichlet}
	For \( n\in\mathbb N \), \( \dMatVec{1,n}\in \{\mat{J}_X,\mat{J}_Y,\mat{J}_D\}^n\), and \( g = M_{\mat{E}_d}^c = \vect{\chi}_{[-\frac{1}{2},\frac{1}{2}]^d} \), the characteristic function of \( [-\tfrac{1}{2},\tfrac{1}{2}]^d \), the scaling functions of de la Vallée Poussin type \( \varphi_{\mat{M}_l}^{\dMatVec{l+1,n}} \), \( l=0,\ldots,n \), yield the Dirichlet kernels \( D_{\mat{M}_l}^{\text{re}} \) from \citep[Section 6]{LangemannPrestin:2010}.
\end{lem}
\begin{proof}
	For \( l=n \) this is evident.
	For \( l=n-1,\,n-2,\ldots, 0 \) we again apply induction over \( l \) and the fact, that the sum in
	\[ B_{\!\!\dMatVec{l+1,n}} = 
	\Biggl(\sum_{\vect{z}\in\mathbb Z^d}\vect{\chi}_{\mathcal Q_d}(\circ-\mat{J}_{l+1}^\T\vect{z})\Biggr)B_{\!\!\dMatVec{l+2,n}}(\mat{J}_{l+1}^{-\T}\circ)
	\]
	consists only of one point or a certain number of points of \( \vect{\chi}_{[-\frac{1}{2},\frac{1}{2}]^d} \) at the boundary. The multiplication following the summation yields \( B_{\!\!\dMatVec{l+1,n}} = \vect{\chi}_{[-\frac{1}{2},\frac{1}{2}]^d}\), \(l=0,\ldots,n \).
\end{proof}
	In case of a dyadic vector \( \dMatVec{1,n} \), it is also possible to derive the corresponding wavelets. If we fix \( l\in\{1,\ldots,n\} \), we obtain, that the elements \( \vect{v}_l \in \pSet{\mat{J}_l^\T} \backslash \{\vect{0}\}\) and \( \vect{w}_l\in\pSet{\mat{J}_l}\backslash\{\vect{0}\} \) are uniquely determined due to \( |\det\mat{J}_l|=2 \). The function \( \tilde B_{\!\!\dMatVec{l,k}}\), \( 1\leq l< k< n \in\mathbb N \) is similarly to \( B_{\!\!\dMatVec{l,k}} \) given by
	\[
			\tilde B_{\dMatVec{l,k}} :=
			\exp(-2\pi\I\circ^\T\vect{w}_l)
			\,
			\sumJMaskOp[\Big]{\modulationOp{\vect{v}_l}g}{B_{\dMatVec{l+1,k}}}{\mat{J}_l}\text{,}
	\]
	where \(\modulationOp{\vect{v}}g := g(\circ - \vect{v}) \).
	It differs from \( B_{\!\!\dMatVec{l,k}} \) just in the first step of the recursive definition, where instead of \( g \), the shifted function \( \modulationOp{\vect{v}_l}g\) is used and a modulation is introduced by the factor \( \E^{-2\pi\I\circ^\T\vect{w}_l} \).
	\begin{definition}\label{def:vdLVPW}
		Let the matrix \( \mat{M}_0 \in \mathbb Z^{d\times d} \) be regular, \( m_0 = |\det\mat{M}_0| > 0 \), a dyadic vector \( \dMatVec{1,n}\), \( n\in \mathbb N \), of matrices be given and denote \( \mat{M}_l\), \(l=1,\ldots,n \), as in Definition~\ref{def:dlVP}.

		The functions \( \psi_{\mat{M}_l}^{\dMatVec{l+1,n}} \), which are defined by their Fourier coefficients as
		\begin{align*}
			c_{\vect{k}}\bigl(\psi_{\mat{M}_l}^{\dMatVec{l+1,n}}\bigr) &= \frac{1}{\sqrt{m_l}}\tilde B_{\!\!\dMatVec{l+1,n}}
			(\mat{M}_l^{-\T}\vect{k}),\quad \vect{k}\in\mathbb Z^{d\times d},\quad l=0,\ldots,n-1\text{,}
		\end{align*}
		are called wavelets of de la Vallée Poussin type.
	\end{definition}
	We introduce the corresponding spaces of their shifts for \( l\in\{0,\ldots,n-1\} \) as
	\[
		W_{\mat{M}_l}^{\dMatVec{l+1,n}}
		:= \operatorname{span}\Bigl\{
		\translationOp{\vect{y}}
		\psi_{\mat{M}_{l}}^{\dMatVec{l+1,n}}
		\ ;\ 
		\vect{y}\in\pSet{\mat{M}_{l}}
		\Bigr\}
		\text{.}
	\]
	\begin{thm}\label{thm:dlVPW}
	For \( n\in\mathbb N \), a regular matrix \( \mat{M}_0\in\mathbb Z^{d\times d} \) and a dyadic vector \( \dMatVec{1,n}\in \bigl(\mathbb Z^{d\times d}\bigr)^n \) of regular matrices, let the scaling functions of de la Vallée Poussin type \( \varphi_{\mat{M}_l}^{\dMatVec{l+1,n}}\), \(l=0,\ldots,n \),
		and the corresponding wavelets of de la Vallée Poussin type \( \psi_{\mat{M}_l}^{\dMatVec{l+1,n}}\), \(l=0,\ldots,n-1\), be given.

		Then, for each \( l\in\{0,\ldots,n-1\} \) the following holds.
		\begin{enumerate}[label=\alph*)]
			\item\label{thm:dlVPW:subSpace}
			\(
			\psi_{\mat{M}_l}^{\dMatVec{l+1,n}} \in
			V_{\mat{M}_{l+1}}^{\dMatVec{l+2,n}}\text{.}
			\)
			\item\label{thm:dlVPW:orth}
			\(
			V_{\mat{M}_{l+1}}^{\dMatVec{l+2,n}} =
			V_{\mat{M}_{l}}^{\dMatVec{l+1,n}} \oplus W_{\mat{M}_{l}}^{\dMatVec{l+1,n}}\text{.}
			\)
			\item\label{thm:dlVPW:Prop}
			The shifts 
			\(
				\translationOp{\vect{y}}\psi_{\mat{M}_{l}}^{\dMatVec{l+1,n}}\text{,}				\ \vect{y}\in\pSet{\mat{M}_l}\text{,}
			\)
			are linearly independent.
			\item\label{thm:dlVPW:Sampling} For \( g\in C^{r}(\mathbb R^d), r\in\mathbb N \),
			we have \( \tilde B_{\dMatVec{l+1,n}}(\mat{M}_l^{-\T}\circ)\in C^r(\mathbb R^d)\).
		\end{enumerate}
	\end{thm}
	\begin{proof}
		\begin{enumerate}[label=\alph*)]
			\item Analogously to the coefficients \( \hat a_{\vect{h}} \) from the proof of Theorem~\ref{thm:dlVP}\,\ref{thm:dlVP:subSpace} we define for \(\vect{h}\in\gSet{\mat{M}_{l+1}^\T}\) the coefficients 
			\begin{equation}\label{eq:twoscaleW}
				\begin{split}
					\hat b_{\vect{h}} :&= \sqrt{|\det\mat{J}_{l+1}|}
					\exp(-2\pi\I\vect{k}^\T\mat{M}_{l}^{-1}\vect{w}_{l+1})
					\bigl(\modulationOp{\vect{v}_{l+1}}g\bigr)^{\mat{J}_{l+1}}(\mat{M}^{-\T}_{l}\vect{h})\\
				&=
					\sqrt{2}
					\exp(-2\pi\I\vect{h}^\T\mat{M}_{l}^{-1}\vect{w}_{l+1})
					\Bigl(\sum_{\vect{z}\in\mathbb Z^d} g(\mat{M}_l^{-\T}\vect{h}+\mat{J}_{l+1}^{-\T}\vect{z} - \vect{v}_{l+1})\Bigr)
					\text{.}
				\end{split}
			\end{equation}
			The statement \ref{thm:dlVPW:subSpace} follows using the same steps as in the proof of Theorem~\ref{thm:dlVP}\,\ref{thm:dlVP:subSpace}, replacing \(B_{\dMatVec{l+1,n}} \) by \( \tilde B_{\dMatVec{l+1,n}} \) and hence obtaining \( \hat b_{\vect{h}} \) instead of \( \hat a_{\vect{h}} \) in the calculations.
			\item The coefficients \( \hat a_{\vect{h}}\), \(\vect{h}\in\gSet{\mat{M}_{l+1}^\T} \), from Theorem~\ref{thm:dlVP}\,\ref{thm:dlVP:subSpace} and \( \hat b_{\vect{h}}\), \(\vect{h}\in\gSet{\mat{M}_{l+1}^\T} \), from \ref{thm:dlVPW:subSpace} fulfill the requirements of Lemma 4.3 in \citep{LangemannPrestin:2010}, more precisely the values
			\begin{equation*}
				\sigma_{\vect{h}} := 
				\exp(-2\pi\I\vect{h}^\T\mat{M}_{l}^{-1}\vect{w}_{l+1})
				,\quad\vect{h}\in\gSet{\mat{M}_{l+1}^\T}\text{,}
			\end{equation*}
		 	even fulfill the requirements of Lemma~\ref{lem:orthNInTranslates}, i.e. if the scaling functions of de la Vallée Poussin type \( \varphi_{\mat{M}_{l}}^{\dMatVec{l+1,n}} \) and \( \varphi_{\mat{M}_{l+1}}^{\dMatVec{l+2,n}} \) have orthonormal shifts, thus does the wavelet \( \psi_{\mat{M}_{l}}^{\dMatVec{l+1,n}} \).
			\item The statement follows directly from \ref{thm:dlVPW:orth} and the linear independence of the shifts \( \translationOp{\vect{x}}\varphi_{\mat{M}_l}^{\dMatVec{l+1,n}}\), \(\vect{x}\in\pSet{\mat{M}_l}\), and \(\translationOp{\vect{y}}\varphi_{\mat{M}_{l+1}}^{\dMatVec{l+2,n}}\), \(\vect{y}\in\pSet{\mat{M}_{l+1}}\), cf. Theorem~\ref{thm:dlVP}\,\ref{thm:dlVP:lu}.
			\item Using the same modifications as in \ref{thm:dlVPW:subSpace}, \( \tilde B_{\dMatVec{l+1,n}}(\mat{M}_l^{-\T}\circ)\in C^r(\mathbb R^d)\) holds analogously to Theorem~\ref{thm:dlVP}\,\ref{thm:dlVP:Sampling}.\qedhere
		\end{enumerate}
	\end{proof}
	\begin{cor}
		For \( n\in\mathbb N \), \( \dMatVec{1,n}\in \{\mat{J}_X,\mat{J}_Y,\mat{J}_D\}^n\) and \( g = M_{\mat{E}_d}^c \) we obtain the Dirichlet wavelets \( \psi_{\mat{N}} \) from \citep[Eq.~(42)]{LangemannPrestin:2010} analogously to Lemma~\ref{lem:SpezialDirichlet}.
	\end{cor}
	The presented construction of the scaling functions
	\( \varphi_{\mat{M}_{l}}^{\dMatVec{l+1,n}} \) and wavelets
	\( \psi_{\mat{M}_{l}}^{\dMatVec{l+1,n}} \) of de la Vallée Poussin type
	introduces a huge variety of periodic anisotropic MRAs: On the one hand the
	function \( g \) can be chosen with very less restrictions, especially it
	can be chosen arbitrarily smooth. Hence the scaling functions and wavelets
	are obtained by sampling an arbitrarily smooth function, e.g. by choosing
	box splines. The directional smoothness of \( g \) does construct a certain
	directional smoothness with respect to the parallelotope defined by
	\(\mat{M}_l^\T\mathcal Q_d\) on each level. On the other hand, the
	construction introduces the possibility to choose any matrix \( \mat{J}_l \)
	in the vector of scaling matrices. This extends the known Dirichlet case
	especially to the shear matrices, e.g. \( \mat{J}_{\vect{Y}}^+
	:= \bigl(\begin{smallmatrix}1&1\\0&2\end{smallmatrix}\bigr)\), but also any
	other integral regular matrix can be chosen. In the construction of the
	dyadic wavelets spanning the orthogonal complement, all matrices of
	determinant \( |\det{\mat{J}_l}|=2 \) are possible.

	Taking a closer look at the coefficients \( \hat a_{l,\vect{h}} \) and
	\( \hat b_{l,\vect{h}} \), that describe the two-scale relation for one
	level of the scaling function and the wavelet respectively, we see
	from~\eqref{eq:twoscale} and~\eqref{eq:twoscaleW}, that both are obtained by
	sampling a certain sum of shifts of \( g \). If \( g \) has finite support,
	these sums also have finite support. Moreover, these coefficients are
	obtained by sampling a function as smooth as \( g \).

	A small extension to this construction, that was omitted in order to keep
	the notation simple, is the possibility to also chose \( g \) for each level
	of the nested spaces separately, i.e. to introduce admissible functions
	\( g_0,\ldots,g_n \) to define the sum in each level. Then of course the
	functions \( B_{\dMatVec{l,k}} , \tilde B_{\dMatVec{l,k}} \) have to be
	adapted, to use \( g_l \) in the operator
	\( \sumJMaskOp{g_l}{B_{\!\!\dMatVec{l+1,n}}}{\mat{J}_l} \). The coefficients
	\( \hat a_{l,\vect{h}} \) and \( \hat b_{l,\vect{h}} \) would each depend on
	\( g_l \) and \( \mat{J}_l \).

	\begin{figure}
		\begin{subfigure}[t]{.5\textwidth}\centering
			\includegraphics{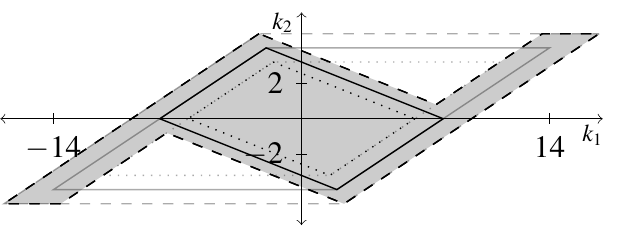}
		\caption[]{
	\( \operatorname{supp} c_{\vect{k}}
	\bigl(\varphi_{\mat{N}}^{(\mat{J})}\bigr)\) and
	\( \operatorname{supp} c_{\vect{k}}\bigl(\varphi_{\mat{M}}^{\emptyset}
	 \bigr)\)}
	 \label{subfig:SuppPhiN}
		\end{subfigure}
		\begin{subfigure}[t]{.5\textwidth}\centering
				\includegraphics{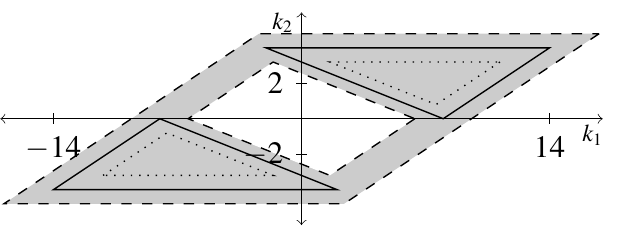}
			\caption{\( \operatorname{supp} c_{\vect{k}}\bigl(\psi_{\mat{N}}^{(\mat{J})}\bigr) \)}\label{subfig:SuppPsiN}
		\end{subfigure}
		\begin{subfigure}[t]{.5\textwidth}\centering
			\includegraphics[width=.975\textwidth]{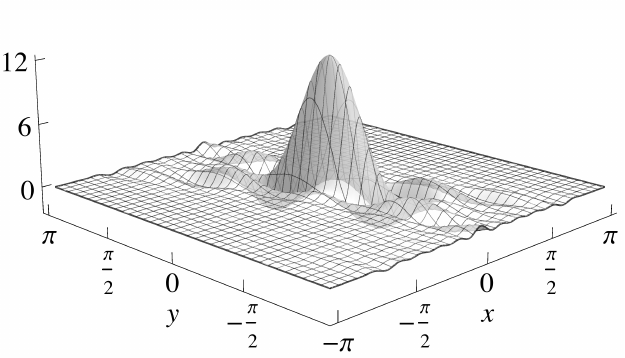}
			\caption{\( \varphi_{\mat{M}}^{\emptyset} \)}\label{subfig:PhiM}
		\end{subfigure}
		\begin{subfigure}[t]{.5\textwidth}\centering
			\includegraphics[width=.975\textwidth]{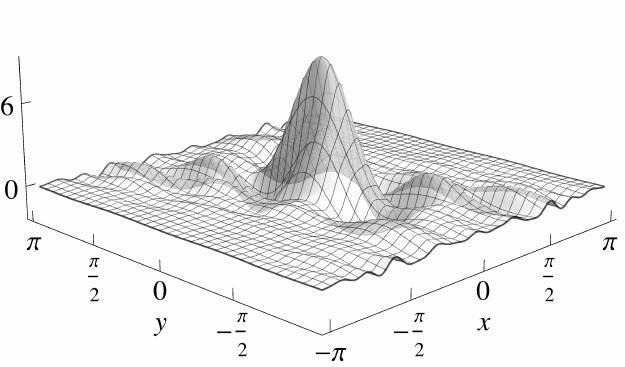}
			\caption{\( \varphi_{\mat{N}}^{(\mat{J})} \)}\label{subfig:PhiN}
		\end{subfigure}
		\begin{subfigure}[t]{.5\textwidth}\centering
			\includegraphics[width=.975\textwidth]{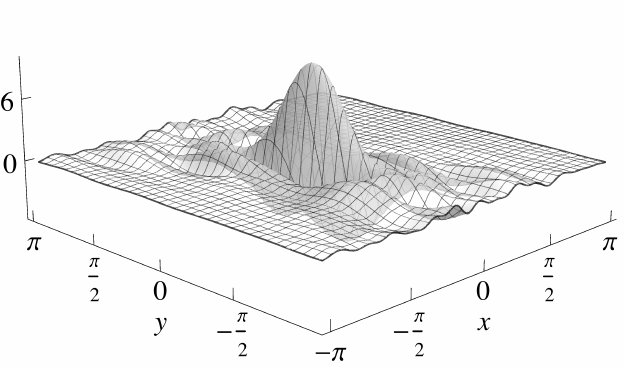}
			\caption{\( \psi_{\mat{N}}^{(\mat{J})} \)
			}\label{subfig:PsiN}
		\end{subfigure}
		\begin{subfigure}[t]{.5\textwidth}\centering
			\includegraphics[width=.975\textwidth]{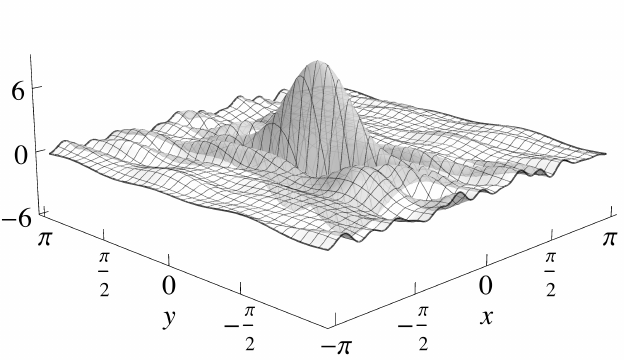}
			\caption{\( \psi_{\mat{N}}^{(\mat{J})} \), for \( \vect{\alpha} = \vect{0} \)}\label{subfig:PsiNDirichlet}
		\end{subfigure}
		\caption[]{The support of the Fourier coefficients \( \operatorname{supp} c_{\vect{k}}\bigl(\varphi_{\mat{N}}^{(\mat{J})}\bigr)\), \( \operatorname{supp} c_{\vect{k}}\bigl(\varphi_{\mat{M}}^{\emptyset}\bigr)\) and \( \operatorname{supp} c_{\vect{k}}\bigl(\psi_{\mat{N}}^{(\mat{J})}\bigr) \) from Example~\ref{ex:dlVP} shown in \subref{subfig:SuppPhiN} and \subref{subfig:SuppPsiN}. The corresponding functions \( \varphi_{\mat{M}}^{\emptyset} \), \( \varphi_{\mat{N}}^{(\mat{J})} \) and \( \psi_{\mat{N}}^{(\mat{J})} \) are plotted in \subref{subfig:PhiM}-\subref{subfig:PsiN}, where finally \subref{subfig:PsiNDirichlet} is constructed setting \( \vect{\alpha}=\vect{0} \) and hence obtaining a wavelet of Dirichlet type, to which \( \psi_{\mat{N}}^{(\mat{J})} \) in \subref{subfig:PsiN} looks more localized. }
	\end{figure}
	\begin{ex}\label{ex:dlVP}
		We look at the decomposition \(%
		\mat{M}%
		:= \bigl(\begin{smallmatrix}16&0\\12&8\end{smallmatrix}\bigr)
		= \mat{J}\mat{N}
		:= \bigl(\begin{smallmatrix}1&1\\0&2\end{smallmatrix}\bigr)\bigl(\begin{smallmatrix}10&-4\\6&4\end{smallmatrix}\bigr)
		\) and use the function \( g=B_{\vect{\alpha}} \), \( \vect{\alpha} = \tfrac{1}{10}(1,1)^\T \), and \( n=1 \). Then, we obtain a sequence of two functions, which consists of \( \varphi_{\mat{M}}^{\dMatVec{2,1}} = \varphi_{\mat{M}}^{\emptyset} \) and \( \varphi_{\mat{N}}^{\mathcal J_{1,1}} = \varphi_{\mat{N}}^{(\mat{J})} \). Both are given by their Fourier coefficients, cf. Definition~\ref{def:dlVP}. In Figure~\ref{subfig:SuppPhiN}, the support of both functions \( g(\mat{M}^{-\T}\circ) \) in gray lines and \( 	\sumJMaskOp[\big]{g}{g}{\mat{J}}(\mat{N}^{-\T}\circ) \) in black lines with gray shade is shown. For both, the dashed line marks the boundary of the support, while the dotted line encircles the area, where the function and hence the coefficients are constantly \( \tfrac{1}{|\sqrt{\det\mat{N}}|}=\tfrac{1}{8} \). The solid line further marks the boundary of the generating set, hence all integer points inside this parallelepiped including the left and lower boundary belong to the generating set \( \gSet[S]{\mat{M}^\T} \). One tenth along that line from any edge, the coefficients equal \( \tfrac{1}{2} \). The maximum value \( c_{\vect{k}}\bigl(\varphi_{\mat{N}}^{(\mat{J})}\bigr) \) on the additional nonzero area outside the parallelepiped is \( \tfrac{1}{4} \) on the solid gray line. Figure~\ref{subfig:SuppPsiN} denotes the support of the corresponding wavelet, restricted to the gray area. Again, both dotted lines mark the plateaus in absolute values of the coefficients.
		
		The corresponding functions \( \varphi_{\mat{M}}^{\emptyset} \), \( \varphi_{\mat{N}}^{(\mat{J})} \) and \( \psi_{\mat{N}}^{(\mat{J})} \) are shown in Figures~\ref{subfig:PhiM}-\subref{subfig:PsiN}. In comparison to the wavelet of de la Vallée Poussin type in Figure~\ref{subfig:PsiN}, a wavelet constructed by using \( g = \chi_{\mathcal Q_d}
		\) is shown in Figure~\ref{subfig:PsiNDirichlet}. This corresponds to a wavelet of Dirichlet type, cf. Lemma~\ref{lem:SpezialDirichlet}, though for the original construction the shear matrix used in this example is not possible. The wavelet of de la Vallée Poussin type is better localized, which can be seen by the flatness of the function away from the origin. This is due to the continuity of the function \(g\), that is continuous for the de la Vallée Poussin type while being a characteristic function of the symmetric unit cube for the Dirichlet case.
	\end{ex}
%
%
	\section{A multiresolution analysis of de la Vallée Poussin type}%
	\label{sec:dlVP-MRA} 
		While the construction from Section~\ref{sec:dlVP} introduces a huge variety of possibilities to choose \( g \) and the scaling matrices \( \mat{J}_l \), \( l=0,\ldots,n \), it introduces the necessity, that a scaling function \( \varphi_{\mat{M}_{l}}^{\dMatVec{l+1,n}} \) or wavelet \( \psi_{\mat{M}_{l}}^{\dMatVec{l+1,n}} \) of de la Vallée Poussin type depends on all scaling matrices \( \mat{J}_l \), the first ones \( \mat{J}_1,\ldots,\mat{J}_l \) in a natural way, because they define \( \mat{M}_l \), but also all following ones, i.e. \( \mat{J}_{l+1},\ldots,\mat{J}_{n} \). This section will introduce a third condition of \( g \) in order to reduce this dependency as far as possible.
	
	For the one-dimensional case from Section~\ref{sec:dlVP1D}, choosing the same function \( g \) is equivalent to setting \( \alpha=\beta \). Theorem~4.1.3 in~\citep{Selig:1998} yields, that these functions constitute an MRA if and only if \( \alpha\leq\tfrac{1}{6} \). If \( \alpha > \tfrac{1}{6} \), Section~\ref{sec:dlVP} does still introduce a construction, though the Fourier coefficients are then constructed from three separated intervals, the support of \( \bigr(\sum_{z\in\mathbb Z} g(\circ+2z)\bigl)g(2^{-1}\circ) \) consists of. Though, the first case is preferable due to it's easier form and single interval of support.
	
	For the multivariate case, we have at least to use \( \varphi_{\mat{M}_l}^{(\mat{J}_{l+1})} \), for \( l<n \), because the first summation is \( g^{\mat{J}_l+1} \) and the matrices \( \mat{J}_l \) may vary even in the dyadic case from one \( l \) to another which does not happen in the one-dimensional case for the factor \( 2 \). In the following, we examine the function \( g \) further and introduce a description of its support, in order to characterize, for which cases the matrix \( \mat{J}_{l+2} \) does not affect \( \varphi_{\mat{M}_l}^{\mathcal J_{l+1,n}} \), which hence simplifies to \(\varphi_{\mathbf{M}_l}^{(\mathbf{J}_{l+1})}\). In order to do that, we look at three successive functions \( B_{\!\!\mathcal J_{k,n}} \), \( k=l,l+1,l+2 \), especially at their support. For simplification, we will first look at the case \( d=n=2 \), where \( \dMatVec{1,2}\in\{\mat{J}_{\text{X}},\mat{J}_{\text{Y}},\mat{J}_{\text{D}}\}^2 \) and discuss the general cases afterwards.
	
	For \( \vect{p}\in\bigl(-\tfrac{1}{2},\frac{1}{2}\bigr)^d \) we define the domain \( \Omega_{\vect{p}}^d \) by
		\[
		\Omega_{\vect{p}}^d := \bigtimes_{j=1}^d \Bigl[-\frac{1}{2}-p_j,\frac{1}{2}+p_j\Bigl]
		\]
	and use the short cut \( \Omega_p^d := \Omega_{p\vect{1}}^d \) for \( p\in \bigl(-\frac{1}{2},\frac{1}{2}\bigr) \). Immediately it follows

	\begin{lem}
		For any admissible function \( g \)  with \( \operatorname{supp} g \subseteq \Omega_{\vect{p}}^d \), \( \vect{p}\in\bigl[0,\tfrac{1}{2}\bigr)^d \) it holds
		\(
		g(\vect{x}) = 1
		\) for all
		\(
		\vect{x}
		\in
		\Omega_{-\vect{p}}^d
		\).
	\end{lem}
	\begin{figure}\centering
			\includegraphics{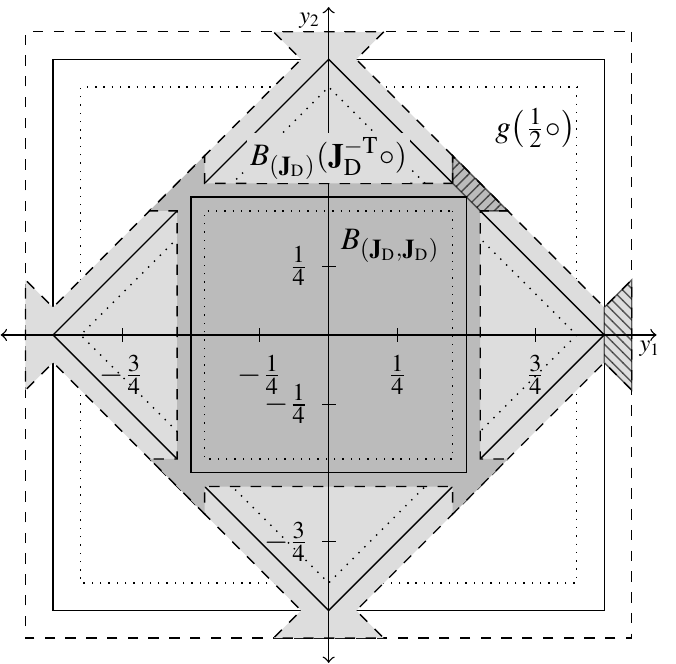}
		\caption{Supports of three successive functions \( B_{\mathcal J_{j,2}} \), \( j=1,2,3 \), where \( \mathcal J_{1,2} = (\mat{J}_{\text{D}},\mat{J}_{\text{D}}) \) and \( \operatorname{supp} g \in \Omega_{\frac{1}{20}}^2 \).}
		\label{fig:completesupport}
	\end{figure}
	A typical situation is depicted in Figure~\ref{fig:completesupport}, where we look at the support of three successive functions in the construction of nested spaces of de la Vallée Poussin type. To obtain independence from the initial matrix \( \mat{M}_0 \) we take the first of the functions \( B_{\!\!\mathcal J_{j,n}} \) unscaled and apply the matrices \( \mat{J}_l, l=k+1,\ldots,n \) to the argument of the function \( B_{\!\!\mathcal J_{j+1,n}} \). We obtain a picture where the effect of summation, i.e. \( g^{(\mat{J}_l)} \) compared to \( g \) itself, is visible for example in the two hatched regions. Here both “inner” functions inherit a support by applying the summation, that only depends on the next upper support. If the support is bigger, than at some point the most inner function would also inherit a certain support that the second inner function obtained from the outer most one by summation. Due to symmetry we restrict the following illustrations again to the first quadrant.
		
	For the rest of this section let \( g \) have the two further properties, that
		\begin{equation}\label{eq:dlVP:TPg}
			g(\vect{x}) = \prod_{i=1}^d g_i(x_i),
			\text{ and }
			\operatorname{supp} g_i \subset \Omega_{p_j}^1
			\quad \text{for } i=1,\ldots,d\text{,}
		\end{equation}
	where each \( g_i:\mathbb R\to\mathbb R\), \(i=1,\ldots,d \), is itself a function having the properties \ref{prop:g:pos} and \ref{prop:g:sum} from Definition~\ref{def:admissible} for \( d=1 \).

	Then we need two auxiliary lemmata.
	\begin{lem}\label{lem:gVereinfacht:X}
		For \( \mat{J}\in\{\mat{J}_X,\mat{J}_Y\} \) and an admissible function \( g:\mathbb R^2\to\mathbb C \) which also fulfills \eqref{eq:dlVP:TPg} and \( \operatorname{supp} g\subset\Omega_p^2 \), \( 0\leq p \leq \tfrac{1}{6} \), it holds
		\begin{equation*}
		\sumJMaskOp{g}{g}{\mat{J}} = g\text{.}
		\end{equation*}
	\end{lem}
	\begin{proof}
		The proof is given for \( \mat{J} = \mat{J}_X \), but can be obtained by the same arguments also for \( \mat{J}=\mat{J}_Y \).

		By assumptions, we have \(\operatorname{supp} g \subset \Omega_p^2\), \(\operatorname{supp} g(\mat{J}_X^{-\T}\circ) \subset
		\mat{J}_X^\T\Omega_p^2\). For \( \vect{z}=(0,z_2)^\T\in\mathbb Z^2\), \( z_1=0 \), we have \( \vect{z}+\Omega_p^2\cap\mat{J}_X\Omega_p^2\neq \emptyset  \) holds for \( 0\leq p < \frac{1}{2} \) if and only if \( z_2\in\{-1,0,1\} \).
		For \( z_2\in\{-1,0,1\} \), the statements \( \forall z_1\in\mathbb Z\backslash\{0\}: (\mat{J}_X^\T\vect{z}+\Omega_p^2)\cap\mat{J}_X\Omega_p^2\neq \emptyset \) and \( p\leq \tfrac{1}{6} \) are equivalent.
		Hence for \( \vect{x}\in\Omega_{(p,-p)^\T}^2\subset \mat{J}_X\Omega_{(p,-p)^\T}^2\) we obtain using  property~\ref{prop:g:sum}
		\[
			\sum_{\vect{z}\in\mathbb Z^2}g(\vect{x}+\mat{J}_X^\T\vect{z}) = g(\vect{x}) = g(\mat{J}_X^{-\T}\vect{x})=1
			\text{.}
		\]
		For \( \vect{x}\in\Omega_{p}^2\), \( |x_2|>\tfrac{1}{2}-p \), the sum over \( \vect{z}\in\mathbb Z^d \) covers a second nonzero index despite \( \vect{z}=\vect{0} \): \( z_2=1 \) for \( x_2<-\tfrac{1}{2}+p \) and \( z_2=-1 \) for \( x_2>\tfrac{1}{2}+p \). Then, we have two cases: Due to \( |x_1|<\tfrac{1}{2}-p\) together with \ref{prop:g:sum} of \( g \) the sum is \( 1 \) and hence
		\[
			\Bigl(\sum_{\vect{z}\in\mathbb Z^2}g(\vect{x}+\mat{J}_X^\T\vect{z})\Bigr)g(\mat{J}_X^{-\T}\vect{x}) = g(\mat{J}_X^{-\T}\vect{x}) = g(\vect{x})\text{.}
		\]
		Further, for \( \tfrac{1}{2}-p\leq |x_1| \leq \tfrac{1}{2}+p \), the summation does not simplify to \( 1 \) due to the dilation caused by \( \mat{J}_X^{-\T} \). It holds
		\begin{align*}
			\Bigl(\sum_{\vect{z}\in\mathbb Z^2}
			g(\vect{x}+\mat{J}_X^\T\vect{z})\Bigr)g(\mat{J}_X^{-\T}\vect{x})
			&= g_1(x_1)\Bigl(g_2(x_2)+g_2(x_2\pm1)\Bigr)g_1\Bigl(\frac{x_1}{2}\Bigr)g_2(x_2)\\
			& = g_1(x_1)g_2(x_2) = g(\vect{x})
			\text{.}
			\qedhere
		\end{align*}
	\end{proof}
	\begin{lem}\label{lem:gSubset:J}
		For \( 0\leq p \leq \tfrac{1}{6}\) let a function \( g:\mathbb R^2\to\mathbb R \) be given as in~\eqref{eq:dlVP:TPg}, having \( \operatorname{supp} g\subset\Omega_p^2 \). Then it holds for \( \mat{J}\in \{\mat{J}_X, \mat{J}_Y,\mat{J}_D\} \), that
		\begin{equation}\label{eq:equal:g:forms1}
			\sumJMaskOp[\big]{g}{  \sumJMaskOp{g}{g}{\mat{J}_D} }{\mat{J}}
			= 
			\sumJMaskOp{g}{g}{\mat{J}}
		\end{equation}
		if and only if
		\begin{equation}\label{eq:subsets:g}
			\Biggl(
			\bigcup_{%
				\vect{z}\in\mathbb Z^2
				} \Omega_p^2+\mat{J}^\T\vect{z}
			\Biggr) \cap \mat{J}^\T\Omega_p^2 \ \ \subset\ \ \mat{J}^\T\mat{J}_D^\T\Omega_{-p}^2\text{.}
		\end{equation}
	\end{lem}
	\begin{proof}
			Writing the operator on the left hand side of~\eqref{eq:equal:g:forms1} we obtain
			\[
			\sumJMaskOp[\big]{g}{    \sumJMaskOp{g}{g}{\mat{J}_D}    }{\mat{J}} = 
			\Bigl(
				\sum_{\vect{z}\in\mathbb Z^2} g(\circ + \mat{J}^\T\vect{z})
			\Bigr)
			\Bigl(
				\sum_{\vect{y}\in\mathbb Z^2} g(\mat{J}^{-\T}\circ + \mat{J}_D^\T\vect{y})
			\Bigr)
			g(\mat{J}_D^{-\T}\mat{J}^{-\T}\circ)\text{,}
			\]
			where both sums are absolutely summable and we may shift the first sum by any \( \mat{J}_D^\T\vect{y}\), \(\vect{y}\in\mathbb Z^2 \), to obtain
			\begin{equation}\label{eq:simplify:g}
				\begin{split}
					\sumJMaskOp[\big]{g}{\sumJMaskOp{g}{g}{\mat{J}_D}    }{\mat{J}}
					&= 
					\Biggl(
						\sum_{\vect{y}\in\mathbb Z^2}
							\Bigl(
								\sum_{\vect{z}\in\mathbb Z^2}g\bigl(\circ + \mat{J}^\T(\vect{z}+\mat{J}_D^\T\vect{y})\bigr)
							\Bigr)
				 			g\bigl(\mat{J}^{-\T}(\circ + \mat{J}^\T\mat{J}_D^\T\vect{y})\bigr)
					\Biggr)
					\\
					&\quad\times g(\mat{J}_D^{-\T}\mat{J}^{-\T}\circ)
					\\[.5\baselineskip]
					&=
					\Biggl(
						\sum_{\vect{y}\in\mathbb Z^2}
						\sumJMaskOp[\big]{g}{g}{\mat{J}}(\circ+\mat{J}^\T\mat{J}_D^\T\vect{y})
					\Biggr)
					g(\mat{J}_D^{-\T}\mat{J}^{-\T}\circ)\text{.}
				\end{split}
			\end{equation}
		The support of the summand \( \vect{y} = \vect{0} \) in the sum is given by
		\[
		\bigcup_{%
		\vect{z}\in\mathbb Z^2
		} \Omega_p^2+\mat{J}^\T\vect{z}\text{.}
		\]
		Let~\eqref{eq:subsets:g} be given. Assume, there exists a value \( \vect{y}\in \mathbb Z^d\backslash\{\vect{0}\} \), such that
		\begin{equation}\label{eq:lem:proof:y-cut}
			\Biggl(\Bigl(
					\mat{J}^\T\mat{J}_D^\T\vect{y} +  \bigcup_{\vect{z}\in\mathbb Z^2} \Omega_p^2+\mat{J}^\T\vect{z}
			\Bigr) \cap \mat{J}^\T\Omega_p^2\Biggr) \ \ \cap\ \ \mat{J}^\T\mat{J}_D^\T\Omega_{-p}^2
		\end{equation}
		is nonempty. Due to \(p\leq \frac{1}{6}\), there exists a point \(\vect{x}\in \mat{J}^\T\mat{J}_D^\T\Bigl(\vect{y}+\Omega_p^2\backslash\Omega_{-p}^2\Bigr) \) in the left intersection. Furthermore \(\mat{J}^\T\mat{J}_D^\T\vect{y}+\mat{J}^\T\mat{J}_D^\T\Omega_p^2\backslash\Omega_{-p}^2\) is not a subset of \(\mat{J}^\T\mat{J}^\T_D\Omega_{-p}^2\) and hence the point \(\vect{x}\) that exists by assumption contradicts~\eqref{eq:subsets:g}. Hence the last summation in~\eqref{eq:simplify:g} simplifies to the summand \( \vect{y}=\vect{0} \). For \( \vect{x}\in\mat{J}^\T\mat{J}_D^\T\Omega_{-p}^2 \) we have the equality \( g(\mat{J}^{-\T}_D\mat{J}^{-\T}\vect{x}) = 1 \). In total, we obtain \( \sumJMaskOp{g}{g}{\mat{J}}(\circ)g(\mat{J}_D^{-\T}\mat{J}^{-\T}\circ) = \sumJMaskOp{g}{g}{\mat{J}} \) and from~\eqref{eq:subsets:g} follows~\eqref{eq:equal:g:forms1}.

		Let~\eqref{eq:equal:g:forms1} be given. Then, the steps from the last paragraph can be applied in reverse order: The equality \( \sumJMaskOp[\big]{g}{    \sumJMaskOp{g}{g}{\mat{J}_D}    }{\mat{J}} =\sumJMaskOp{g}{g}{\mat{J}} \) is equivalent to \( \sumJMaskOp{g}{g}{\mat{J}}(\circ)g(\mat{J}_D^{-\T}\mat{J}^{-\T}\circ) = \sumJMaskOp{g}{g}{\mat{J}} \). For \( \vect{y}\neq \vect{0} \) the sets in~\eqref{eq:lem:proof:y-cut} are empty. For \( \vect{y}=\vect{0} \) we also obtain from~\eqref{eq:lem:proof:y-cut}, that for
		\[
		\vect{x} \in 
		\Biggl(\bigcup_{\vect{z},\ \|\vect{z}\|_{\infty}\leq1} \Omega_p^2+\mat{J}^\T\vect{z}
		\Biggr) \cap \mat{J}^\T\Omega_p^2\quad\text{it holds}\quad g(\mat{J}^{-\T}_D\mat{J}^{-\T}\vect{x})=1 \text{,} \]
	which completes the proof of the equivalence of~\eqref{eq:equal:g:forms1} and~\eqref{eq:subsets:g}.
		\end{proof}
	\begin{thm}
			\label{thm:dlVP:unabhaengig}
		For \( 0\leq p\leq \tfrac{1}{14} \) let an admissible function \( g:\mathbb R^2\to\mathbb R \) be given as in~\eqref{eq:dlVP:TPg}, having \( \operatorname{supp} g\subset\Omega_p^2 \). Let \( \mat{M}_0\in\mathbb Z^{2\times 2} \) be a regular matrix and \( \dMatVec{1,n}\in \{\mat{J}_X, \mat{J}_Y, \mat{J}_D\}^n \), \( n\in \mathbb N \), be a matrix vector. Then, the following equalities hold for the corresponding scaling functions \( \varphi_{\mat{M}_l}^{\dMatVec{l+1,n}} \) of de la Vallée Poussin type, \( l=0,\ldots,n \). 
		\begin{enumerate}[label=\alph*)]
			\item\label{thm:dlVP:unabhaengig:Skalierung} If \( \mat{J}_n \in\{\mat{J}_X,\mat{J}_Y\} \), then 
				\(
				\varphi_{\mat{M}_{n-1}}^{\dMatVec{n,n}} = \varphi_{\mat{M}_{n-1}}^{\emptyset}\text{,}
				\)
				\item\label{thm:dlVP:unabhaengig:Rotation} If \( \mat{J}_n = \mat{J}_D \), then
				\(
				\varphi_{\mat{M}_{n-2}}^{\dMatVec{n-1,n}} = \varphi_{\mat{M}_{n-2}}^{\dMatVec{n-1,n-1}}\text{.}
				\)
		\end{enumerate}
	\end{thm}
	\begin{proof}
		Applying Lemma~\ref{lem:gVereinfacht:X}
		for \( \mat{J}_n\in\{\mat{J}_X,\mat{J}_Y\} \) to any scaling function from Definition \ref{def:dlVP} yields for their Fourier coefficients \( c_{\vect{k}}(\varphi_{\mat{M}_{n-1}}^{\dMatVec{n,n}})\), \(\vect{k}\in\mathbb Z^2 \), that
		\begin{equation*}
		\begin{split}
			c_{\vect{k}}(\varphi_{\mat{M}_{n-1}}^{\dMatVec{n,n}})
			&=	\frac{1}{\sqrt{m_{n-1}}}\sumJMaskOp{g}{g}{\mat{J}_n}(\mat{M}_{n-1}^{-\T}\vect{k})
			&=	\frac{1}{\sqrt{m_{n-1}}}g(\mat{M}_{n-1}^{-\T}\vect{k})
			&=	c_{\vect{k}}(\varphi_{\mat{M}_{n-1}}^{\emptyset})\text{.}
		\end{split}
		\end{equation*}
		For \( \mat{J}_n=\mat{J}_D \), the Fourier coefficients are given by
		\begin{align*}
			c_{\vect{k}}(\varphi_{\mat{M}_{n-1}}^{\dMatVec{n,n}})
			&=	\frac{1}{\sqrt{m_{n-1}}}\sumJMaskOp{g}{g}{\mat{J}_D}(\mat{M}_{n-1}^{-\T}\vect{k}),\quad\vect{k}\in\mathbb Z^2,
			\intertext{ and }
			c_{\vect{k}}(\varphi_{\mat{M}_{n-2}}^{\dMatVec{n-1,n}})
			&=	\frac{1}{\sqrt{m_{n-1}}}\sumJMaskOp[\bigl]{g}{   \sumJMaskOp{g}{g}{\mat{J}_D} }{\mat{J}_{n-1}}
				(\mat{M}_{n-2}^{-\T}\vect{k}),\quad\vect{k}\in\mathbb Z^2\text{.}
		\end{align*}
		 Due to symmetry, the next steps are only described in the first quadrant, cf. Figures \ref{subfig:SFSup20:XD} and \subref{subfig:SFSup20:DD}, and we again omit the case \( \mat{J}_{n-1}=\mat{J}_{\text{Y}} \).
		\begin{figure}[tbp]
			\setlength{\belowcaptionskip}{1.5\baselineskip}
				\begin{subfigure}[b]{\textwidth}\centering
					\includegraphics{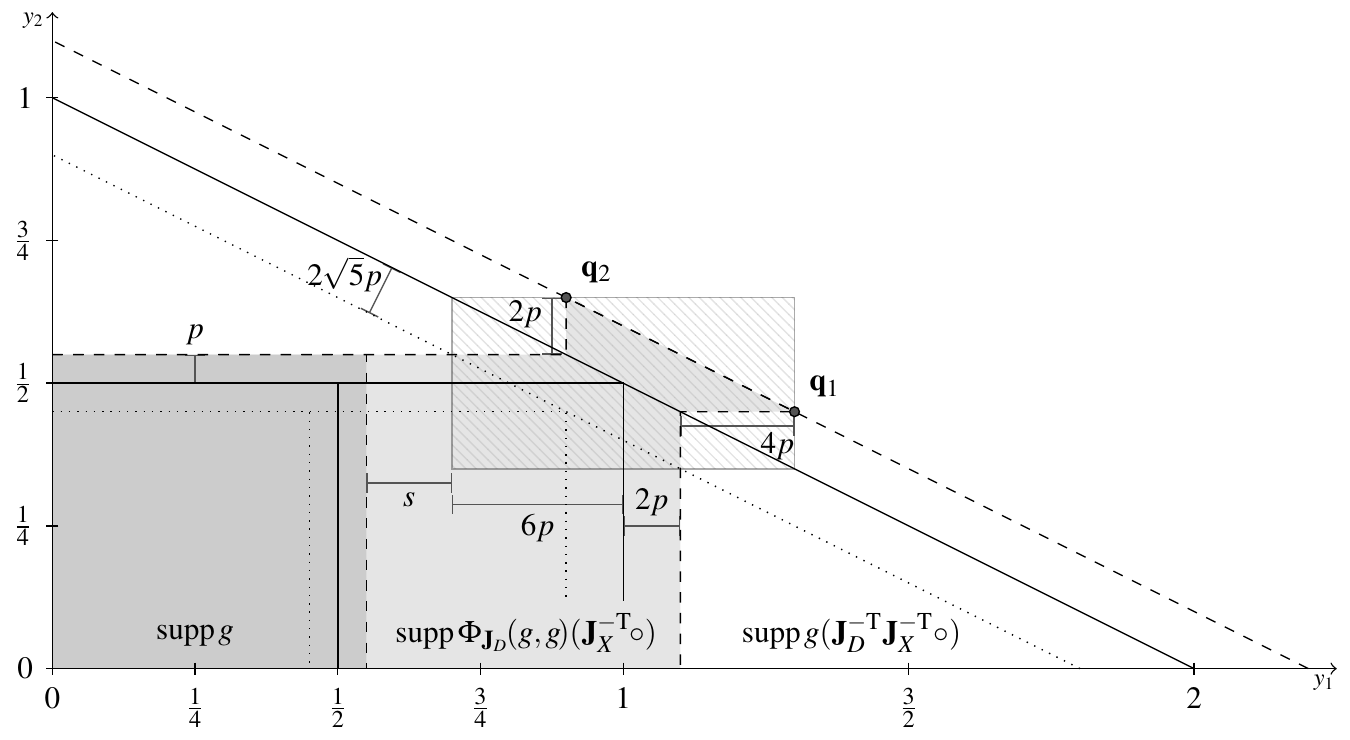}
				\captionsetup{skip=0\baselineskip}
				\caption{\( \mat{J}_{n-1}=\mat{J}_X \) and \( \mat{J}_n = \mat{J}_D \)}\label{subfig:SFSup20:XD}
			\end{subfigure}
			\begin{subfigure}[b]{\textwidth}\centering
					\includegraphics{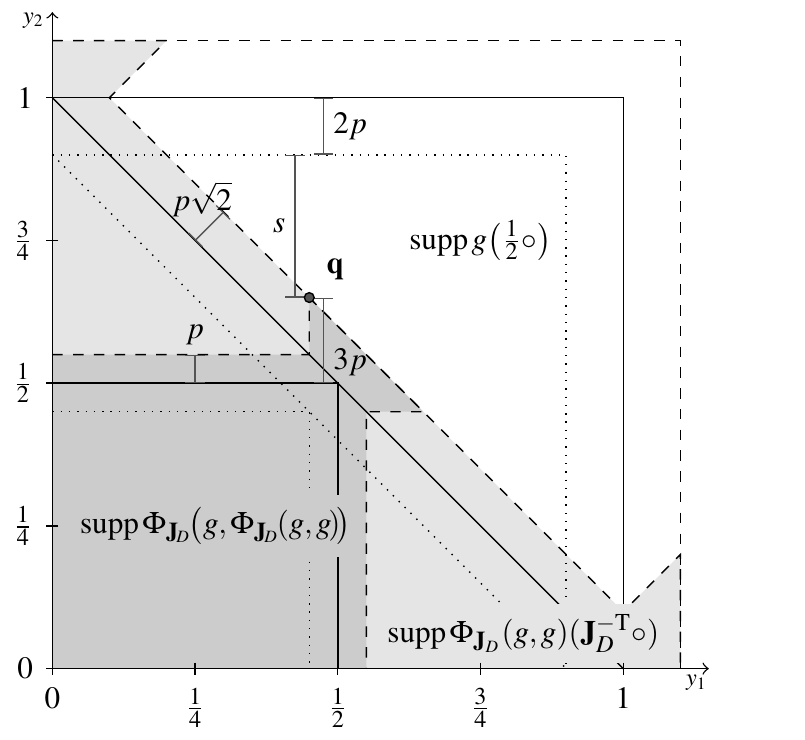}
				\captionsetup{skip=0\baselineskip}
				\caption{\( \mat{J}_{n-1}= \mat{J}_n = \mat{J}_D \)}\label{subfig:SFSup20:DD}
			\end{subfigure}
			\caption{Illustrations of the supports of three successive scaling functions, \( p=\tfrac{1}{20} \).}
		\end{figure}

		Let \( \mat{J}_{n-1}=\mat{J}_X \), cf. Figure~\ref{subfig:SFSup20:XD}. The points \( \vect{r}_x = \bigl(\tfrac{1}{2}-p,\tfrac{1}{2}-p\bigr)^\T\) and \( \vect{r}_y = \bigl(-\tfrac{1}{2}+p,\tfrac{1}{2}-p\bigr)^\T \) on the boundary of \( \Omega_p^2 \) fulfill \(\mat{J}_X^\T\mat{J}_D^\T\vect{r}_x = (2-4p,0)^\T \) and \(\mat{J}_X^\T\mat{J}_D^\T\vect{r}_y = (0,1-2p)^\T\), hence \(\sumJMaskOp{g}{g}{\mat{J}_D}(\mat{J}_X^{-\T}\circ) \) can be restricted to lie inside \( \Omega_{(6p,3p)}^2 \). It holds further, that
		\[
		\operatorname{supp} g \cap [1-6p,1+6p]\times\bigl[\tfrac{1}{2}-p,\tfrac{1}{2}+p\bigr] = \emptyset\text{.}
		\]
		Hence, we can apply Lemma~\ref{lem:gVereinfacht:X} if the intersection of this area with \( g^{\mat{J}_X} \) is empty. This area is shown in Figure~\ref{subfig:SFSup20:XD}) as the hatched area. If this intersection is empty, we obtain \(\sumJMaskOp{g}{g}{\mat{J}_D}(\mat{J}_X^{-\T}\circ)=g(\mat{J}_X^{-\T}\circ)\). Using \( \operatorname{supp} g \subset \Omega_p^2 \) the emptiness holds if and only if for the distance \( s \) shown in Figure~\ref{subfig:SFSup20:XD} it holds
		\(
		1-6p - \bigl(\tfrac{1}{2}+p\bigr) =: s \geq 0,\) which is equivalent to \(p\leq \frac{1}{14}\). Hence it holds
		\begin{align*}
			c_{\vect{k}}(\varphi_{\mat{M}_{n-2}}^{\dMatVec{n-1,n}})
			&=\frac{1}{\sqrt{m_{n-1}}}\sumJMaskOp[\big]{g}{ \sumJMaskOp{g}{g}{\mat{J}_D} }{\mat{J}_X}(\mat{M}_{n-2}^{-\T}\vect{k})\\
			&=\frac{1}{\sqrt{m_{n-1}}}\sumJMaskOp{g}{g}{\mat{J}_X}(\mat{M}_{n-2}^{-\T}\vect{k})\\
			&= c_{\vect{k}}(\varphi_{\mat{M}_{n-2}}^{\dMatVec{n-1,n-1}})\text{.}
		\end{align*}
		For \( p\leq \tfrac{1}{14} \), the requirements of Lemma~\ref{lem:gVereinfacht:X} are fulfilled and hence we obtain further \( \varphi_{\mat{M}_{n-2}}^{\dMatVec{n-1,n-1}} = \varphi_{\mat{M}_{n-2}}^{\emptyset} \).

		For \( \mat{J}_{n-1}=\mat{J}_D \), the supports for \( p=\tfrac{1}{20} \) of three successive scaling functions are shown in Figure~\ref{subfig:SFSup20:DD}. For the support \( \operatorname{supp} B_{\!\!\dMatVec{n-1,n}} = \operatorname{supp} \sumJMaskOp{g}{ \sumJMaskOp{g}{g}{\mat{J}_D} }{\mat{J}_D} \) let \( \vect{q} \) be the corner point having the maximal \( y \) coordinate. At \( q_{1} = \frac{1}{2}-p \), the horizontal boundary of \( \operatorname{supp} g \subset \Omega_p^2 \) coincides with the diagonal, which marks the boundary of \( \operatorname{supp} \mat{J}_D\Omega_0^2 \), where a part of the function \( g \) is present, shifted by \(\mat{J}_D^\T(0,1)^\T \). Perpendicular to this horizontal line, we have a line from \( \bigl(\tfrac{1}{2}-p,\tfrac{1}{2}+p\bigr)^\T \) to the boundary of the aforementioned support. Its end point is given by \( \vect{q} = \big(\tfrac{1}{2}-p,\tfrac{1}{2}+3p\bigr)^\T \). 

		For this point \( \vect{q} \), we have \( \vect{q}\in\mat{J}_D^\T\mat{J}_D^\T\Omega_{-p}^2 = 2\Omega_{-p}^2 = \bigl\{\vect{x}\ ;\  \|\vect{x}\|_{\infty} \leq 1-2p\bigr\}\) if and only if 
		\[
		\|\vect{q}\|_{\infty} = \frac{1}{2}+3p \leq 1-2p\text{,}\quad\text{which means }
		1-2p-\biggl(\frac{1}{2}+3p\biggr) =:s \geq 0,\text{ i.e. }
		p \leq \frac{1}{10}\text{.}
		\]
		This distance \( s \) is also shown in Figure~\ref{subfig:SFSup20:DD}. The line segment exists if and only if \( \operatorname{supp}  B_{\!\!\dMatVec{n-1,n}} \subset \mat{J}_D^\T\mat{J}_D^\T\Omega_{-p}^2\) holds. Hence by Lemma~\ref{lem:gSubset:J} we have for \( \vect{k}\in\mathbb Z^2 \) that
			\begin{align*}
				c_{\vect{k}}(\varphi_{\mat{M}_{n-2}}^{\dMatVec{n-1,n}})
				&=\frac{1}{\sqrt{m_{n-1}}}\sumJMaskOp[\big]{g}{ \sumJMaskOp{g}{g}{\mat{J}_D}  }{\mat{J}_D}(\mat{M}_{n-2}^{-\T}\vect{k})\\
				&=\frac{1}{\sqrt{m_{n-1}}}\sumJMaskOp{g}{g}{\mat{J}_D}(\mat{M}_{n-2}^{-\T}\vect{k})\\
				&= c_{\vect{k}}(\varphi_{\mat{M}_{n-2}}^{\dMatVec{n-1,n-1}})\text{.}\qedhere
			\end{align*}
	\end{proof}
	%
	%
	%
	For the case \( d>2 \) we can prove a similar statement applying the same arguments as before: We denote by \( \mathcal J_d \) the set of matrices containing \( \mat{J}_{x_i} \), \( i\in\{1,\ldots,d\} \), as a matrix, which scales the \( x_i \) axis by \( 2 \) and the matrices \( \mat{J}_{x_i,x_j} \), \( i\neq j \), \( i,j\in\{1,\ldots,d\} \) as the rotation by \( \tfrac{\pi}{4} \) in the \((x_i x_j) \)-plane, which scales this plane by \( \sqrt{2} \)\ at the same time. In other words \( \mathcal J_d \) is a generalization of the previously used matrices, especially we have \( \mathcal J_2 = \{\mat{J}_{\text{X}},\mat{J}_{\text{Y}},\mat{J}_{\text{D}}\} \).
	\begin{lem}\label{lem:dlVP:unabhaengigD}
		Let \( d>2 \), \( 0\leq p\leq\tfrac{1}{14}\) and an admissible function \( g: \mathbb R^d\to\mathbb C \) be given as in~\eqref{eq:dlVP:TPg} with \( \operatorname{supp} g \subset \Omega_p^d \). Further, let  a regular matrix \( \mat{M}_0\in\mathbb Z^{d\times d} \) and a vector of matrices \( \dMatVec{1,n}\in {\mathcal J_d}^n \), \( n\in\mathbb N \), be given, which fulfills the following statement
		\begin{equation}\label{eq:dgr2Restr}
			\mat{J}_l = \mat{J}_{x_i,x_j} \Rightarrow \mat{J}_{l-1}\in\{\mat{J}_{x_i},\mat{J}_{x_j},\mat{J}_{x_i,x_j}
			\},\quad i,j\in\{1,\ldots,d\},\ l=2,\ldots,n\text{.}
		\end{equation}
		Then, the scaling functions \( \varphi_{\mat{M}_l}^{\dMatVec{l+1,n}} \), \( l=0,\ldots,n \), of de la Vallée Poussin type fulfill
		\begin{enumerate}[label=\alph*)]
			\item\label{lem:dlVP:unabhaengigD:Skalierung} for \( \mat{J}_n  = \mat{J}_{x_i}\), \( i\in\{1,\ldots,d\} \), that 
			\(
			 	\varphi_{\mat{M}_{n-1}}^{\dMatVec{n,n}} = \varphi_{\mat{M}_{n-1}}^{\emptyset}\text{,}
			\)
			\item\label{lem:dlVP:unabhaengigD:Rotation} for \( \mat{J}_n = \mat{J}_{x_i,x_j} \), \( i\neq j \), \( i,j\in\{1,\ldots,d\} \), that 
			\(
				\varphi_{\mat{M}_{n-2}}^{\dMatVec{n-2,n}} = \varphi_{\mat{M}_{n-2}}^{\dMatVec{n-1,n-1}}\text{.}
			\)
		\end{enumerate}	
	\end{lem}
	\begin{proof}
		Statement \ref{lem:dlVP:unabhaengigD:Skalierung} is the higher dimensional formulation of Theorem~\ref{thm:dlVP:unabhaengig}\,\ref{thm:dlVP:unabhaengig:Skalierung} and its proof follows directly by the same arguments. The second part follows from the fact, that~\eqref{eq:dgr2Restr} restricts all argumentations to the \( (x_i x_j) \)-plane and hence the same steps as in Theorem~\ref{thm:dlVP:unabhaengig}\,\ref{thm:dlVP:unabhaengig:Rotation} apply.
	\end{proof}
%
%
	\section{Example}\label{sec:Example} 
	\begin{figure}
		\enlargethispage{1.5\baselineskip}
		\setlength{\belowcaptionskip}{.65\baselineskip}
		\begin{subfigure}[b]{.5\textwidth}\centering
			\includegraphics[width=\textwidth]{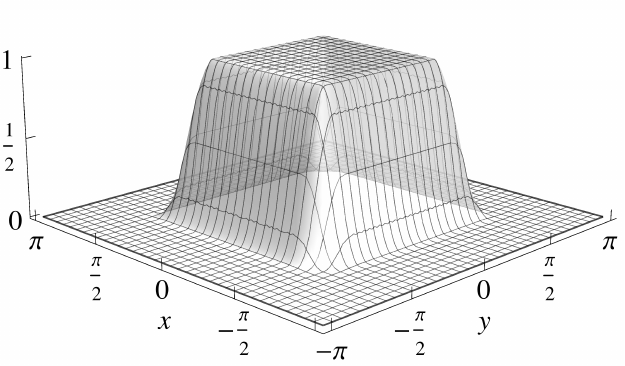}
			\caption{\( M_{\Theta}^c \), \( \Theta = \frac{\pi}{8}\bigl(\begin{smallmatrix}
				8&0&1&0&-1\\
				0&8&1&0&1
			\end{smallmatrix}\bigr) \) }\label{subfig:MTheta}
		\end{subfigure}
		\begin{subfigure}[b]{.5\textwidth}\centering
			\includegraphics{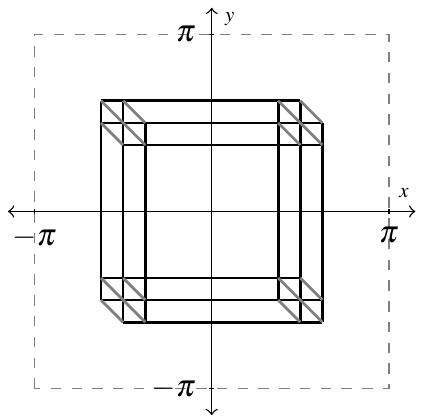}
			\caption{discontinuities of \( M_{\Theta}^c \)}\label{subfig:discontinuities}
		\end{subfigure}
		\begin{subfigure}[b]{.5\textwidth}\centering
			\includegraphics{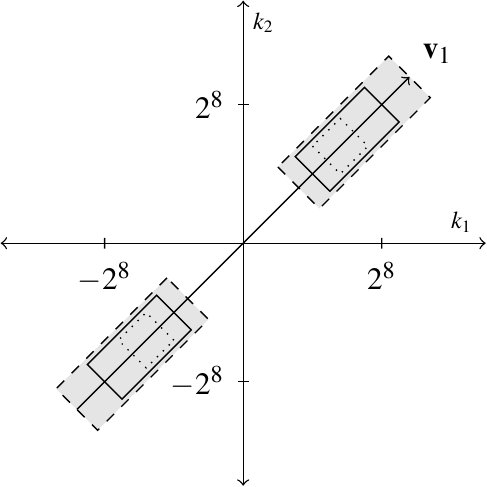}
 			\caption{frequency support of \( \psi_1 \)}\label{subfig:ckPsi1}
 		\end{subfigure}
 		\begin{subfigure}[b]{.5\textwidth}\centering
 			\includegraphics{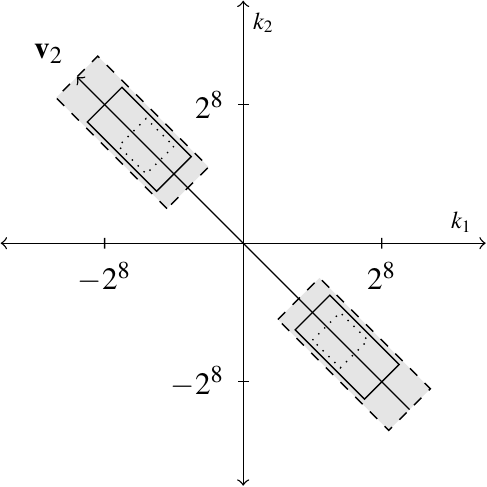}
			\caption{frequency support of \( \psi_2 \)}\label{subfig:ckPsi2}
		\end{subfigure}
		\begin{subfigure}[b]{.5\textwidth}\centering
			\includegraphics{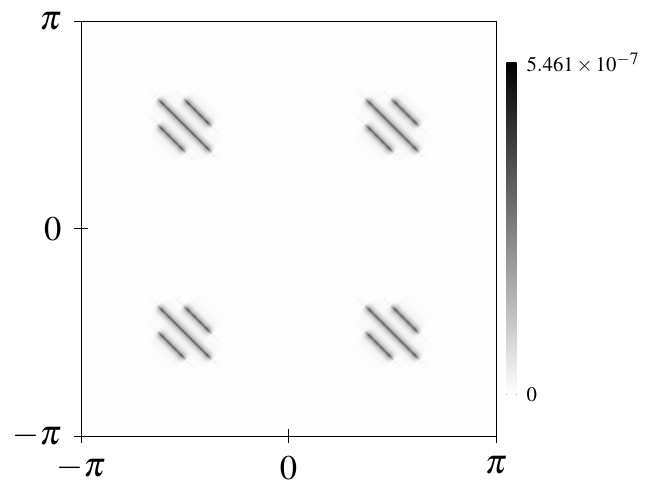}
			\caption{fraction \( g_{\mat{N}_1} \) for \( M_{\Theta}^c \) in \( V_{\mat{N}_1}^{\psi_1} \subset V_{\mat{M}_1}^{\varphi_{\mat{M}_1}^{\emptyset}}\)
			}\label{subfig:gM1}
		\end{subfigure}
		\begin{subfigure}[b]{.5\textwidth}\centering
			\includegraphics{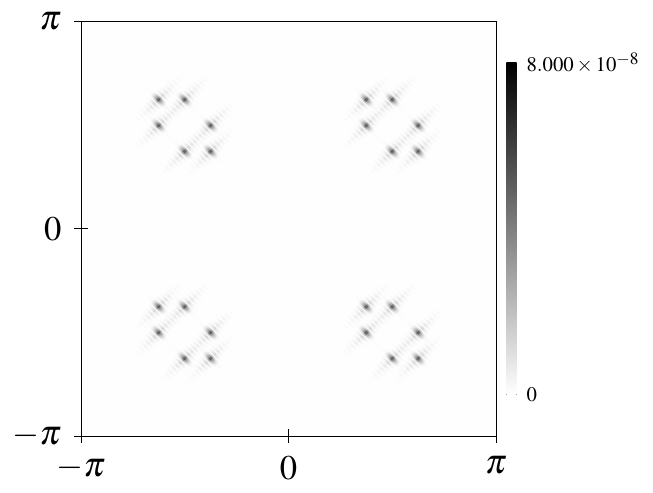}
			\caption{fraction \( g_{\mat{N}_2} \) for \( M_{\Theta}^c \) in
				\( V_{\mat{N}_2}^{\psi_2} \subset %
				V_{\mat{M}_2}^{\varphi_{\mat{M}_2}^{\emptyset}}\)}%
			\label{subfig:gM2}
		\end{subfigure}
		\captionsetup{skip=-2pt}
		\caption{%
		The box spline \( M_{\Theta}^c \) in \subref{subfig:MTheta} has discontinuities in its second (black) and third (gray) directional derivatives, orthogonal to the lines shown in \subref{subfig:discontinuities}. Sampling and decomposing with \( \mat{M}_k=\mat{J}_k\mat{N}_k\), \(k=1,2 \), we obtain two wavelets \( \psi_1 \)  and \(\psi_2 \), whose Fourier coefficients are samplings of the functions shown in \subref{subfig:ckPsi1} and \subref{subfig:ckPsi2}. The wavelet parts \( g_{\mat{N}_1} \in V_{\mat{N}_1}^{\psi_1} \) and \( g_{\mat{N}_2} \in V_{\mat{N}_2}^{\psi_2} \) are shown in \subref{subfig:gM1} and \subref{subfig:gM2}.}
	\end{figure}
	Consider the centered box spline \( \mat{M}_{\Theta}^c \), \( \Theta = \frac{\pi}{8}\bigl(\begin{smallmatrix}
		8&0&1&0&-1\\
		0&8&1&0&1
	\end{smallmatrix}\bigr) \), which is shown in Figure~\ref{subfig:MTheta}. It is a two times differentiable function along any line on the torus \( \mathbb T^d \). Along its second and third directional derivatives, there are discontinuities, which are depicted in Figure~\ref{subfig:discontinuities}.
	
	Then, the following steps are performed using the software package~\citep{Be13url}, which was written in \emph{Mathematica 9}. It contains implementations for generating scaling functions and wavelets, sampling arbitrary functions on patterns \(\pSet[S]{\mat{M}}\) for regular matrices \(\mat{M}\in\mathbb Z^{d\times d}\), performing a change of basis from the interpolatory basis of \(V_{\mat{M}}^{\emptyset}\) into the translates of a de la Vallée Poussin type scaling function \(\varphi_{\mat{M}}\)--- which of course depend on an admissible function \(g\) ---, performing both the Fourier transform on arbitrary patterns \(\pSet{\mat{M}}\), and the wavelet decomposition for an arbitrary chain of de la Vallée Poussin type wavelets, and displaying the result.
	
	We take a look at the sampling obtained from different matrices \( \mat{M}_1 = \bigl(\begin{smallmatrix} 512&512\\-64&64	\end{smallmatrix}\bigr)\) and \( \mat{M}_2 = \bigl(
	\begin{smallmatrix}
		64&64\\-512&512
	\end{smallmatrix}
	\bigr) \), their decompositions \( \mat{M}_1 = \mat{J}_{1}\mat{N}_1 \) and \( \mat{M}_2 = \mat{J}_{2}\mat{N}_2\), where we choose \( \mat{J}_1 = \mat{J}_{\text{X}} \), \( \mat{J}_2 = \mat{J}_{\text{Y}} \), and \( g=B_{\vect{\alpha}}\), \(\vect{\alpha} = \tfrac{1}{10}(1,1)^\T \). This yields two different wavelets \( \psi_1 \) and \( \psi_2 \). Their frequency supports are shown in Figures \ref{subfig:ckPsi1} and \subref{subfig:ckPsi2}. For \( k\in\{1,2\} \) the dashed line marks the boundary of the support of \( \tilde B_{(\mat{J}_k)}(\mat{N}_k^{-\T}\circ) \), the dotted lines encircle the plateau of the wavelet. For both, the direction through the center of their plateaus is drawn as \( \vect{v}_k \). If we sample \( M_{\Theta}^c \) on the pattern \( \pSet{\mat{M}_k} \) and use the fundamental interpolant of \( V_{\mat{M}_k}^{\varphi_{\mat{M}_k}^{\emptyset}} \), we obtain an approximation \( f_{\mat{M}_k} \) in this space, cf. \citep{BergmannPrestin:2013}.
	
	Using the wavelet transform, we obtain the lines of discontinuities of the third directional derivatives analogously to \citep[Section 6.2]{Bergmann:2013}, where a similar box spline was examined with wavelets from the Dirichlet case. With the de la Vallée Poussin setting the lines do posses less artifacts and their amplitude is higher. While the first decomposition in the mentioned Section 6.2 of \citep{Bergmann:2013} was only able to detect both diagonals at the same time with the one level of decomposition, this example also demonstrates the ability to look at the diagonal lines separately. We see that orthogonal to \( \vect{v}_1 \) lie all 12 parallel lines of discontinuities in the third directional derivative. They are obtained when looking at \( g_{\mat{N}_1} \in V_{\mat{N}_1}^{\psi_1}\), see Figure~\ref{subfig:gM1}). The only discontinuities of the third directional derivative along \( \vect{v}_2 \) are the 24 singular points at the ends of the lines from the previous case. The corresponding function is shown in Figure~\ref{subfig:gM1}).
	
	In fact, both wavelets can also be obtained by choosing two different factorizations of the matrix \( \mat{M} = 
	\bigl(\begin{smallmatrix}
		1024&0\\0&1024
	\end{smallmatrix}\bigr) \), i.e. the idea of having an \( 1024\times1024 \) pixel image of the box spline \(M_{\Theta}^c\). These factorizations are \( \mat{M} = \mat{J}_{\text{D}}\mat{J}_{\text{Y}}\mat{J}_{\text{Y}}\mat{J}_{\text{Y}}\mat{M}_1 \) and \( \mat{M} = \mat{J}_{\text{D}}\mat{J}_{\text{X}}\mat{J}_{\text{X}}\mat{J}_{\text{X}}\mat{M}_2 \). For the second decomposition in each of these factorizations, Theorem \ref{thm:dlVP:unabhaengig} cannot be applied due to \( \alpha_j > \frac{1}{14}, j=1,2\). Though the following factors again yield, that just one dilation matrix \( \mat{J}_l \) is needed to construct the scaling function and wavelet of de la Vallée Poussin type.
%
%
	\section{Conclusion} 
	In this paper, we examined a characterization of a periodic anisotropic multiresolution analysis \( \bigl(\{\mat{J}_k\}_{k>0},\{V_j\}_{j\geq 0}\bigr) \) in order to introduce and investigate multivariate scaling functions of de la Vallée Poussin type \( \varphi_j \) that posses a certain decay in their Fourier coefficients. The decay is given by an admissible function \( g \) which can be chosen quite generally. Especially for the one-dimensional case, certain functions \( g \) resemble the Dirichlet and Féjer kernel and all de la Vallée Poussin means. In the multivariate case, the Dirichlet kernels are also a special case of the presented construction.
	
	With a set of regular matrices \( \mat{J}_1,\ldots,\mat{J}_n \), these scaling functions of de la Vallée Poussin type yield a finite sequence of nested shift-invariant spaces. For the dyadic case, i.e. where all dilation matrices \( \mat{J}_l \) are of determinant \( 2 \), the framework also yields a similar construction for the wavelets that form the orthogonal complements between the nested spaces by their shifts. When decomposing a single function \( f \) into fractions in these wavelet spaces, we obtain directional information about \( f \).

In the construction, there are only a few restrictions on the function \(g\), the sequence of wavelets \( \psi_{\mat{M}_l}^{(\mat{J}_l)} \), \( l=1,\ldots,n \), is based on. We introduced the notion of admissibility for \(g\), i.e. having compact support and being a partition of unity. Extending the presented construction, one could also introduce a sequence \(g_j\) of such admissible functions and perform the construction using a different admissible function for each scaling function. How smoothness properties of the function \(g\) characterize the localization properties of the sequence of wavelets, is a topic for future research.
	
	While in the general construction, the scaling functions of de la Vallée Poussin type \( \varphi_{\mat{M}_l}^{\mathcal J_{l+1,n}} \), \( l=0,\ldots,n \), depend on a complete vector \( \mathcal J_{l+1,n} = \bigl(\mat{J}_{l+1},\ldots,\mat{J}_n\bigr) \) of matrices, Theorem \ref{thm:dlVP:unabhaengig} and Lemma~\ref{lem:dlVP:unabhaengigD} examine this vector in more detail. In particular, for a further restriction to \( g \), i.e. having a certain compact support, we obtain the identity \( \varphi_{\mat{M}_l}^{\dMatVec{l+1,n}} = \varphi_{\mat{M}_l}^{\dMatVec{l+1,l+1}} = \varphi_{\mat{M}_l}^{(\mat{J}_{l+1})} \) using the set of matrices from Dirichlet case in \citep[Section 6]{LangemannPrestin:2010}. For the dyadic case we also obtain the corresponding wavelets \( \psi_{\mat{M}_l}^{(\mat{J}_l)} \), \( l=1,\ldots,n \). In this setting, the MRA may also contain any finite sequence of shearing matrices \( \mat{J}_{Y}^{\pm},\mat{J}_{X}^{\pm} \) and their higher dimensional generalizations. Then, the matrix vector \( \mathcal J_{j+1,n} \) in each definition of a wavelet chain is longer than just \( 1 \), but still finite for any scaling function \( \varphi_j \).

These functions can now be used to examine a broad range of directional decompositions of a given function \( f \). On the one hand, for a given matrix \( \mat{M} \) on whose pattern \( 2\pi\pSet{\mat{M}} \) the function was sampled, there is a huge variety of matrices to decompose \( \mat{M} \), even for just the dyadic case. These decompositions are given by any factorization of \( \mat{M} \) into a product of matrices with determinant \( 2 \). On the other hand, for a given preference of one ore more directions, these functions of de la Vallée Poussin type give rise to many MRAs that prefer this set of “directions of interest” in their wavelet spaces. This enables a huge variety of applications towards anisotropic image decomposition with fast algorithms.

\end{document}